\documentclass[]{amsart}
\usepackage[utf8]{inputenc}
\usepackage{graphicx}
\usepackage{amsmath,amssymb,mathtools,amsthm,amscd,amsfonts}
\usepackage[numbers,sort&compress]{natbib}
\usepackage{tikz}
\usepackage[svgnames]{xcolor} 
\usepackage{hyperref}
\usepackage{url}
\usepackage{dsfont}
\usepackage{stmaryrd}
\usepackage[]{algorithm2e}
\usepackage{algpseudocode}

\usepackage{tikz}
\usetikzlibrary{arrows.meta,positioning}

\DeclareMathAlphabet\mathbfcal{OMS}{cmsy}{b}{n}

\hypersetup{
	colorlinks=true,       
	linkcolor=cyan,          
	citecolor=green,        
	filecolor=magenta,      
	urlcolor=cyan           
}

\numberwithin{equation}{section}
\numberwithin{figure}{section}

\allowdisplaybreaks
\theoremstyle{plain}
\newtheorem{lemma}{Lemma}[section]
\newtheorem{proposition}[lemma]{Proposition}

\newtheorem{theorem}[lemma]{Theorem}

\theoremstyle{definition}
\newtheorem{definition}[lemma]{Definition}

\newcommand{\Z}{\mathbb{Z}}
\newcommand{\N}{\mathbb{N}}

\newcommand{\ind}[1]{{\mathds{1}_{\{#1\}} }}
\newcommand{\dd}{\textnormal{d}}

\renewcommand{\P}{\mathbb{P}}
\newcommand{\R}{\mathbb{R}}
\newcommand{\E}{\mathbb{E}}

\newcommand{\bX}{{\bf{X}}}
\newcommand{\bV}{{\bf{V}}}
\newcommand{\bv}{{\bf{v}}}

\newcommand{\bw}{{\bf{w}}}
\newcommand{\bu}{{\bf{u}}}
\newcommand{\bx}{{\bf{x}}}

\newcommand{\bp}{{\bf{p}}}
\newcommand{\by}{{\bf{y}}}
\newcommand{\bb}{{\bf{b}}}

\newcommand{\bd}{{\bf{d}}}

\newcommand{\bz}{{\bf{z}}}
\newcommand{\bM}{{\bf{M}}}
\newcommand{\be}[1]{ {\bf{e}}_{{#1}} }

\newcommand{\cB}{{\mathcal{B}}}
\newcommand{\cP}{{\mathcal{P}}}
\newcommand{\cA}{{\mathcal{A}}}

\newcommand{\cS}{{\mathcal{S}}}
\newcommand{\balpha}{{\boldsymbol{\alpha}}}
\newcommand{\bbeta}{{\boldsymbol{\beta}}}

\newcommand{\bfeta}{{\boldsymbol{\eta}}}
\newcommand{\bxi}{{\boldsymbol{\xi}}}

\newcommand{\btheta}{{\boldsymbol{\theta}}}
\newcommand{\HG}{\mathrm{HG}}
\newcommand{\obv}{{\bfeta}}

\newcommand{\bxe}[1]{ {\mathbf{x}}({#1}) }
\newcommand{\berno}[1]{\mathrm{Ber}({#1})}

\newcommand{\MSC}[1]{\small \textbf{MSC:} {#1}}

\newcommand{\keywordsbeginning}[1]{\small  \textbf{Keywords:} {#1}}

\title[Bernstein duality for multi-type Wright--Fisher processes]{
Multi-type $\Lambda$-Wright--Fisher processes with polynomial drift: Bernstein duality, long-term behaviour, and simulation of moments}
\author{Sebastian Hummel}
\address{BOKU University, 
Institute of Mathematics, Department of Natural Sciences and Sustainable Resources,
Gregor-Mendel-Stra\ss{}e 33/II, 
1180 Vienna, 
Austria}
\email{sebastian.hummel@boku.ac.at}

\date{\today}

\begin{document}

\maketitle

\begin{abstract}
    The $(d+1)$-type $\Lambda$-Wright--Fisher process with polynomial drift is a L\'evy-type process on the $d$-dimensional unit simplex
    whose generator entails a particular jump and diffusive component, 
    as well as a polynomial drift component.
    In this work, the process is connected to a (Bernstein) dual process, which admits an interpretation as an urn model with coloured balls.
    Exploiting this duality, we identify various parameter regimes: one in which the $\Lambda$-Wright--Fisher process is absorbed almost surely at a vertex of the simplex, and others in which it admits a unique nontrivial ergodic probability measure.
    The results include verifiable criteria for absorption versus persistence,
    as well as representations of absorption and stationary measures.
    Moreover, the duality relation is exploited to devise and implement a simulation-based approach for estimating moments of the $\Lambda$-Wright--Fisher process.
\end{abstract}
\MSC{Primary: 60J25, 60J90, 60J95; Secondary: 92D15}

\keywordsbeginning{Wright--Fisher process, branching--coalescing structures, duality, absorption probabilities, ergodic measures}

\tableofcontents

\section{Introduction}

$\Lambda$-Wright--Fisher processes have attracted substantial interest in both theoretical and applied probability~\cite{Evans1996,Favero2021,Overbeck1995}.
One key motivation is modelling type-frequency dynamics of alleles, traits, or genes in populations subject to evolutionary forces such as random mating, selection, and mutation.
Inferring from genetic data the extent to which these forces act on a population has driven advances in statistical inference methods and in the underlying probability theory.
But interest in these processes also stems from their intimate connection to rich mathematical objects such as coalescent-, branching-, and branching--coalescing-structures.
These objects underpin many inference tools in population genetics because they model genealogical relationships within the population~\cite{wake2009}.

The formal connection between the branching--coalescing structure and the $\Lambda$-Wright--Fisher process is typically a duality between two stochastic processes.
Generally, duality between two stochastic processes allows certain functionals of one process to be expressed in terms of the other.
Ideally one of the processes is simpler to analyse, enabling the study of the more complex process.
Beyond applications in population genetics, when the duality function is rich enough to determine the law of the process, the relation has been used to establish well-posedness of martingale problems~\cite{Ethier1986, Depperschmidt2024}.
Moreover, duality has found applications in mathematical finance~\cite{Cuchiero2012, Filipovic2016, Filipovic2020,AbiJaber2024}.

Certain $\Lambda$-Wright--Fisher processes admit a duality in which weighted moments of the primal process are expressed via a dual process~\cite{Etheridge2009,Etheridge2010}.
This duality is useful in both theory and applications.
A challenge in dealing with the dual process is that its transition rates involve the stationary distribution of the primal process.
In some cases that distribution is explicitly known and easy to compute, allowing the dual to be used for simulations of the Wright--Fisher process;
by a change of measure these simulations extend to a larger class~\cite{jenkins2017}.
However, there remain Wright--Fisher processes for which this simulation approach fails.
If the objective is solely to obtain moments of the forward process,
a more direct approach is to solve the associated partial integro-differential equation for the moment-generating or characteristic function and differentiate to recover moments.
In practice, closed-form solutions are uncommon, and numerical methods are complicated by boundary behaviour on the simplex: any simulation requires care to preserve the positivity and unit-sum constraint.

Another example of duality is that of Bernstein form, which has originally been introduced for the one-dimensional setting~\cite{Cordero2022}.
The duality function is a polynomial expressed in Bernstein basis polynomials, and the dual process takes values in the coefficients of that representation.
It is a special instance of a function-valued dual~\cite{dawson2014spatial}.
Although the dual's state space is more complex than in the weighted moment duality~\cite{Etheridge2009,Etheridge2010}, its transition rates do not require knowledge of the primal process's stationary distribution.
Bernstein duality has been used in the one-dimensional setting to represent absorption probabilities and ergodic measures~\cite{Cordero_2026} for $\Lambda$-Wright--Fisher processes with polynomial drifts.
Using Bernoulli factories to handle drift functions beyond the polynomial setting, 
the Bernstein duality has been extended to other models such as the Allen--Cahn equation \cite{Koskela2024} and the spatial Wright--Fisher process \cite{koskela_well_posedness_2026}, 
albeit always in the two-type setting only.

The class of finite-dimensional $\Lambda$-Wright--Fisher processes that we want to consider here,
have a certain polynomial drift term.
This is to mean that the component of the generator associated to the deterministic part of the process is a finite-degree polynomial (of form~\eqref{eq:drift_form}).
We remark that, in general, the corresponding processes do not necessarily agree with the class of \emph{polynomial processes} introduced in~\cite{Cuchiero2012,Filipovic2016}, 
and they extend the processes treated in~\cite{GonzalezCasanova2020}.

The purpose of this paper is fivefold. 
First, we extend Bernstein duality from the one-dimensional setting to the finite-dimensional case, thereby closing an open problem posed in~\cite{Cordero2022}.
Second, we give conditions characterising when the finite-dimensional process is absorbed at the boundary and when it admits a stationary distribution that places mass also in the interior.
Third, in the case when the process is absorbed, we derive formulas for absorption probabilities.
Fourth, in the case when a stationary distribution exists that places mass also in the interior of the state space, 
we derive representations of limiting distribution, in particular for models with drift terms corresponding to so-called parent-dependent mutation kernels, 
a special case of direct applied interest.
Fifth, we explore implementations of the duality to compute moments of non-neutral $\Lambda$-Wright--Fisher processes based on the Bernstein dual process, which differs from the exact simulation in~\cite{jenkins2017} or Runge--Kutta methods.

The paper is organised as follows. 
Main results are presented in Section~\ref{sec:main_results},
which we conclude with a list of open problems.
Section~\ref{sec:proofDuality} contains the proof of the duality. 
Section~\ref{sec:longterm} gathers the proofs concerning the long-term behaviour of the $\Lambda$-Wright--Fisher processes.
Finally, Appendix~\ref{sec:generator_identity} contains the (lengthy but straightforward) computation of the generator identity between the $\Lambda$-Wright--Fisher process and its dual.

\bigskip

\textit{Notation:}
Throughout the paper, $d\in \N$ is fixed.
We use boldface notation for vectors. 
The $d$-dimensional simplex is defined as 
\[
\Delta_d\coloneqq\Big\{\bx=(\bxe{1},\bxe{2},\ldots,\bxe{d})\in [0,1]^{d}: \textstyle\sum_{i=1}^d  \bx(i)\leq 1 \Big\}.
\] 
For $\bx\in \Delta_d$, we set $\bxe{d+1}\coloneqq 1-\sum_{i=1}^d  \bxe{i}$.
We write $\mathrm{int}(\Delta_{d})$ for the interior of~$\Delta_{d}$, 
and for $i\in \llbracket d+1\rrbracket$ set $\partial_{i}\Delta_{d}\coloneqq\{\bx\in \Delta_d:\bx(i)=0\}$.
For $n,m\in \Z$ with $n\leq m$,
the set of integers from $n$ to $m$ is denoted by 
$\llbracket n,m\rrbracket\coloneqq\{k\in \Z: n\leq k\leq m\}$; if $n=1$ we just write $\llbracket m\rrbracket\coloneqq\llbracket 1, m \rrbracket$. 
For $\balpha\in\N_0^{d+1}$, 
we use the notation $\bx^\balpha\coloneqq \prod_{i=1}^{d+1} \bx(i)^{\balpha(i)}$, 
and we stress that this includes the dependent $\bx$-component with index~$d+1$.
For $n\in \N_0$, we write $\Pi(n)\coloneqq \{\balpha\in \N_0^{d+1}: \lvert \balpha \rvert = n\}$, 
where $\lvert \balpha\rvert \coloneqq \sum_{i=1}^{d+1}\balpha(i)$; 
that is, $\Pi(n)$ is the set of weak compositions of~$n$ into~$d+1$ parts (weak composition means that a~$0$-part is allowed).
Write $\be{i}$ for the $i$th unit vector in $\N_0^{d+1}$, so it has value~$1$ at index~$i$ and~$0$ elsewhere; 
in particular, $n\cdot\be{i}$ is the vector with value~$n$ at index~$i$ and~$0$ elsewhere.
When $\be{i}$ appears as a point of $\Delta_d$, 
we mean the vertex of $\Delta_d$ with $\bx(i)=1$; 
in particular, $\be{d+1}$ then corresponds to $0\in\R^d$.
Let $\bM(n)$ denote the set of mappings from $\Pi(n)\to \R$; and set $\bM\coloneqq \cup_{n\in \N_0} \bM(n)$.
For $\bv\in \bM$, define $\dim(\bv)\coloneqq n$ if $\bv\in \bM(n)$.
Set $\lVert \bv\rVert_{\infty}=\max_{\balpha}\lvert \bv(\balpha)\rvert$.
For $i\in \llbracket d+1 \rrbracket$, define $\obv(i;n)\in \bM(n)$ as the map that assigns~$1$ to $n\cdot\be{i}$ and is~$0$ elsewhere; if $n=1$, we just write $\obv(i)$ to shorten notation.
For $\balpha\in \N_0^{d+1}$ and $n\in \N$ with $n\leq \lvert \balpha\rvert$, 
the $\N_0^{d+1}$-valued random variable $\HG^\balpha_n$ is said to have multivariate hypergeometric distribution with parameters $\balpha$ and $n$ if it assigns to $\bbeta\in \N_0^{d+1}$ with $\lvert \bbeta \rvert =n$ the probability
\begin{equation}\label{eq:hypergeometric}
    \mathfrak{P}(\HG^\balpha_n=\bbeta) = \tbinom{\lvert \balpha\rvert}{\lvert \bbeta\rvert}^{-1}\prod_{i=1}^{d+1}\tbinom{\balpha(i)}{\bbeta(i)},
\end{equation}
and zero else.
We write $\mathfrak{E}$ for the corresponding expectation.
For $\balpha\in \Pi(n)$, let $\binom{n}{\balpha}\coloneqq n!/(\prod_{i=1}^{d+1} (\balpha(i)!))$  denote the multinomial coefficient.

\section{Main results} \label{sec:main_results}
Fix $d\in \N$.
We are concerned with the $(d+1)$-type version of the $\Lambda$-Wright--Fisher process with a specific drift component.
This jump-diffusion process is defined via its generator that acts on a twice continuously differentiable function~$f$ on~$\Delta_d$ as
\begin{equation}
\cA f(\bx)=\cA_{\Lambda(0)}f(\bx)+\cA_{\Lambda_{>}}f(\bx)+\cA_{\bd}f(\bx), \label{eq:wfgenerator}
\end{equation} with the building blocks
    
\begin{align}
	\cA_{\Lambda(0)}f(\bx)\coloneqq&\ \frac{\Lambda(\{0\})}{2}\sum_{i,j=1}^d \bxe{i}\Big(\ind{j}(i)-\bxe{j}\Big)\frac{\partial^2 f}{\partial \bxe{i}\partial \bxe{j}} ,\label{eq:wfterm} \\
	\cA_{\Lambda_{>}}f(\bx)\coloneqq&\ \int_{(0,1]}  \sum_{i=1}^{d+1}\bx(i) \left[f\left((1-r)\bx+r\be{i}\right)-f\big(\bx\big)\right]\frac{\Lambda(\dd r)}{r^2},\label{eq:lambdaterm} \\
	\cA_{\bd}f(\bx)\coloneqq&\ \sum_{i=1}^{d} \bd(\bx;i)\frac{\partial f}{\partial \bxe{i}},\label{eq:driftterm} 
\end{align}

where $\Lambda$ is from the finite Borel measures on $[0,1]$, which we abbreviate by $\mathcal{M}_f([0,1])$ and which gives the process its name, and $\bd(\bx;i)$ is the $i$th drift component.
We will only consider drifts taking the form
\begin{equation}\label{eq:drift_form}
    \bd(\bx;i)=\sum_{m=1}^{\kappa} \bb(m) \sum_{\bbeta\in \Pi(m)} \left(\bp(\bbeta;i)-\frac{\bbeta(i)}{m}\right) \binom{m}{\bbeta} \bx^{\bbeta},
\end{equation}
where $\kappa\in \N$, $\bb \in \R_{\geq}^{\kappa}$, 
and for $m\in\llbracket \kappa\rrbracket$, $\bbeta\in \Pi(m)$ and $i\in \llbracket d+1\rrbracket$, 
we have $\bp(\bbeta;i)\in [0,1]$.
To be exact, we will work under the following two additional (and for now technical) conditions. 
First, $\textstyle\sum_{i=1}^{d+1} \bp(\bbeta;i)=1$. 
Second, for $\lvert \bbeta\rvert>1$ and for all $i\in \llbracket d+1 \rrbracket$,
\begin{equation}
    \text{if } \bbeta(i)=0 \ \Rightarrow \  \bp(\bbeta;i)=0\label{eq:cond_probability_0}.
\end{equation}
We write $P\coloneqq(\bp(\be{i};j))_{i,j\in \llbracket d+1 \rrbracket}$.
For $j\in \llbracket d+1\rrbracket$, define $\btheta(j) \coloneqq \min_{i\in \llbracket d+1\rrbracket} P_{i,j}$ so that $P_{i,j}=\btheta(j)+\tilde{P}_{i,j}$ for some $\tilde{P}_{i,j}$;
set then $\tilde{\bp}(\be{i};j)\coloneqq\tilde P_{i,j}$.
We stress that $\btheta(j)$ is allowed to be~$0$.
It follows from the conditions on $\bp$ that $\lvert \btheta\rvert \in [0,1]$.
While the condition and parametrisation of the drift may appear unusual at this stage,
we anticipate that they will become more intuitive once we introduce the dual process in Section~\ref{sec:Bernstein_dual_process}.

That each part of~$\cA$ separately generates a Feller semi-group is well-known~\cite{shiga1981,Bertoin2003}; that also $\cA$ does will be established in Theorem~\ref{thm:Bernsteinduality} below.
We write $\cP$ for a generic set of Wright--Fisher process parameters $\kappa\in \N$, $\bb \in \R_{\geq}^{\kappa}$, $\bp(\bbeta;i)\in [0,1]$ for $\bbeta\in \N_0^{d+1}$ and $i\in \llbracket d+1\rrbracket$, and $\Lambda \in \mathcal{M}_f([0,1])$.
In particular, this is meant to imply that the parameters of such $\cP$ satisfy the conditions around~\eqref{eq:cond_probability_0}.

\subsection{Bernstein coefficient process as a Bernstein dual process} \label{sec:Bernstein_dual_process}
Given a set of Wright--Fisher process parameters $\cP$, we construct a process of an urn that contains balls of potentially~$d+1$ different colours.
To this end, consider first the Markovian process $L=(L_r)_{r\geq 0}$ on~$\N_0$ 
whose transition rates may not yet seem intuitively linked to the preceding discussion, except that they are also expressed in terms of the parameters in~$\cP$.
For completeness, these rates are as follows.
$$n\mapsto\begin{cases}n-k+1 & \text{at rate } \binom{n}{k}\lambda_{n,k}, \ k\in \llbracket 3,n\rrbracket, \\
n-1 & \text{at rate }n\lvert \btheta\vert  \bb(1)+ \binom{n}{2}\lambda_{n,2}, \\
n+m-1 & \text{at rate }n\bb(m)(1-\ind{m=1} \lvert \btheta\rvert ),\ m\geq 1,
\end{cases}$$
where 
		\begin{equation}
			\lambda_{n,k}\coloneqq \int_{[0,1]} z^{k-2}(1-z)^{n-k}\Lambda(\dd z), \label{defrateneutjump}
		\end{equation}
		with the convention that $0^0=1$; in particular $\lambda_{n,2}=\Lambda(\{0\})+\int_{(0,1]} (1-z)^{n-2}\Lambda(\dd z)$.

\medskip 

Fix~$T>0$ and assume $L_0=n$.
Given $L=(L_r)_{r\in [0,T]}$, we construct the urn process in a time interval $[0,T]$ as follows.
The number of balls in the urn at time~$t$ is set to $\hat{L}_{t}\coloneqq L_{T-t}$.
In particular, the urn process starts in general with a random number $\hat{L}_0=L_T$ coloured balls (at time~$t=0$).
At this stage we did not specify how many balls of what colour are in the urn.
What will be prescribed, however, is how the composition of the urn changes at jumps of $\hat{L}$.
This includes silent transitions, which is to mean that $n\mapsto n$.

If the ball-count decreases by $\tilde{m}$ balls, $\tilde{m}>0$, then sample just before the decrease $\tilde{m}+1$ balls uniformly at random (without replacement); 
this is always possible because a decrease of~$\hat{L}$ is induced by an increase of $L$, which requires there to be at least still one ball in the urn so that there are indeed at least~$\tilde{m}+1$ balls to sample from.
A sample can be summarised by $\bbeta\in\Pi(\tilde{m}+1)$, with $\bbeta(i)$ coding the number of balls of colour~$i$ in the sample.
For a sample~$\bbeta$, the~$\tilde{m}+1$ balls making up that sample are replaced by a single ball of colour~$i$ with probability $\bp(\bbeta;i)$; leading to the decrease by $\tilde{m}$.

Whenever the number of balls increases by just one ball to a total of~$n$ balls, there are two possible mechanisms the composition can change.
Either sample the colour of the additional ball such that the new ball has colour $i$ with probability $\btheta(i)/\lvert \btheta\rvert$; or sample one ball uniformly at random before the increase, 
put back the sampled ball together with a ball of the same colour.
The first way of increase occurs with probability 
$n\lvert \btheta\rvert\bb(1)/(n\lvert \btheta\rvert\bb(1)+\binom{n}{2}\lambda_{n,2})$,
the other way with the complement probability.
If there is a silent jump of~$\hat{L}$, then we sample one ball uniformly at random;
if the colour of that ball is~$i$, 
we replace it by a ball of colour~$j$ with probability $\tilde P_{ij}/(1-\lvert\btheta\rvert)$.
Whenever the ball-count increases by two or more balls, we sample just before the increase one ball uniformly at random, and all the new balls are deemed to carry the colour of the sampled ball; the sampled ball is returned to the urn.
In particular, in light of such transitions, the condition~\eqref{eq:cond_probability_0} just means that in a decrease, 
the ball replacing the sample can be interpreted to be from within the sample.

\smallskip

The state of an urn with~$n$ balls can be coded by an element of $\Pi(n)$.
Let $\bV_r\in \bM(L_r)$ be a probability measure on the possible urn compositions at time $T-r$, i.e. a probability measure on~$\Pi(\hat{L}_{T-r})$.
Conditional on $\bV_0$ being fixed to some~$\bv\in\bM(L_0)$, 
the measure $\bV_T$ gives the probability distribution on urn compositions at time~$0$ that are compatible under the dynamic with an urn composition at time~$T$ being distributed as~$\bv$.

Extending the probability measures arising under such an urn dynamic to general finite signed measures is the idea behind the dual process.
In particular, the dual process takes values in the finite signed measures on urn compositions arising under such a dynamic.
For more details in the case~$d=1$, in particular for the relationship of that process to genealogical structures underlying the interpretation in population genetics, 
we refer to~\cite{Cordero2022} and~\cite{Cordero_2026}.

The following linear operators, which we refer to as Bernstein operators, will later be used to prescribe the change of measures induced by transitions of the urn process.
For now, we only state their definition.

\begin{definition}[Bernstein operators]\label{def:bramutoperators}
Fix the $\Lambda$-Wright--Fisher process parameters~$\cP$.
Recall the definition of the multivariate hypergeometric distribution $\HG_{\cdot}^{\cdot}$ from~\eqref{eq:hypergeometric}.
Given~$\cP$, define for all $n\in \N$ and $\bv\in \bM(n)$, the following operators.
\begin{enumerate}
    \item For $k\in\llbracket 2, n\rrbracket $, 
    set $C^{n,k}: \bM(n)\to\bM(n-k+1)$ as 
    \begin{equation}
        C^{n,k}(\bv)\coloneqq \left( \mathfrak{E}\Big[\bv\big(\balpha+(k-1)\, \HG^{\balpha}_1\big)\Big]:\, \balpha\in \Pi(n-k+1)\right).\label{eq:bernstein_diff_operator}
    \end{equation}
    \item For $m\in \llbracket 2, \kappa \rrbracket$, 
    set  $D^{n,m}: \bM(n)\to\bM(n+m-1)$ as  
    \begin{equation}
        D^{n,m}(\bv)\coloneqq \left( \mathfrak{E}\Big[\sum_{i=1}^{d+1}\bp(\HG^{\balpha}_m;i)\, \bv(\balpha-\HG^{\balpha}_m+\be{i})\Big]:\, \balpha\in \Pi(n+m-1)\right),\label{eq:bernstein_drift_operator}
    \end{equation}
    \item Set $D^{n,1}: \bM(n)\to\bM(n)$ as follows.
    If $\lvert\btheta\rvert =1$, 
    set $D^{n,1}(\bv)\coloneqq\bv$; otherwise
    \begin{equation}
        D^{n,1}(\bv)\coloneqq \left( \mathfrak{E}\bigg[\sum_{i=1}^{d+1}\frac{\tilde{\bp}(\HG^{\balpha}_1;i)}{1-\lvert \btheta\rvert}\, \bv(\balpha-\HG^{\balpha}_1+\be{i})\bigg]:\,\balpha\in \Pi(n)\right).\label{eq:bernstein_drift_operator_1}
    \end{equation}
    \item Set $D^{n,0}: \bM(n)\to\bM(n-1)$ as follows: 
    If $\lvert\btheta\rvert =0$, set $D^{n,0}(\bv)\coloneqq0$; otherwise
    \begin{equation}
        D^{n,0}(\bv)\coloneqq \left( \sum_{i=1}^{d+1} \frac{\btheta(i)}{\lvert \btheta\rvert}\bv(\balpha+\be{i}) : \balpha\in \Pi(n-1)\right).\label{eq:bernstein_drift_operator_0}
    \end{equation}
\end{enumerate}
\end{definition}
It is readily checked that these operators are indeed linear and admit a matrix representation.
We exploit this fact for the simulations discussed in Section~\ref{sec:example}.

For $\bv\in \bM(n)$, we let $\mathbb{O}_{\cP}$ be the set of operators from $\bM(n)$ to $\bM(k)$, for some $k \in \N_0$, 
that are the composition of a finite number of compatible Bernstein operators, including the identity operator of $\bM(n)$ corresponding to an empty composition. 
Compatibility refers not only to the necessity for dimensions to align appropriately but also to the following: 
If $\mathcal{P}$ is such that $\bb(1)=0$, 
then $D^{\cdot,0}$ as well as $D^{\cdot,1}$ will be deemed incompatible;
the transitions in $L$ that would trigger such change of measure would not occur.
Let $\cS_{\cP}(\bv) \subset \bM$ be defined by 
\begin{equation}
\cS_{\cP}(\bv)\coloneqq\left\{O (\bv)\in \bM: \ O\in \mathbb{O}_{\cP}\right\}. \label{defcvv}
\end{equation}
The set $\cS_{\cP}(\bv)$ is by definition invariant by composition of the associated compatible operators $C^{\cdot,\cdot}$ and $D^{\cdot,\cdot}$.
Since $\mathbb{O}_{\cP}$ is countable, also $\cS_{\cP}(\bv)$ is countable for every fixed~$\bv$ and we can equip it with the discrete topology.
In particular, we will consider convergence in distribution relative to that topology. 

The following definition of the Bernstein coefficient process makes rigorous the previous heuristic describing it as the evolution of tracking measures in an urn process with coloured balls.

\begin{definition}[Bernstein coefficient process]\label{def:BCP}
	Let $\bxi\in\bM$ and $\bv\in \cS_{\cP}(\bxi)$. 
	The Bernstein coefficient process starting at $\bv$ is the $\cS_{\cP}(\bxi)$-valued continuous-time Markov chain $\bV\coloneqq(\bV_r)_{r\geq 0}$ such that $\bV_0=\bv$ and that transitions from $\bw\in \cS_{\cP}(\bxi) \cap \bM(n)$ to
	\begin{itemize}
		\item $C^{n,k}(\bw)\in \bM(n-k+1)$, $k\in \llbracket 2, n\rrbracket$, 
        at rate $\tbinom{n}{k}\lambda_{n,k}$.
		\item $D^{n,m}(\bw)\in  \bM(n+m-1)$, $m\in \llbracket \kappa\rrbracket $, at rate $n \bb(m)(1-\ind{m=1}\lvert \btheta\rvert)$.
        \item $D^{n,0}(\bw)\in  \bM(n-1)$ at rate $n\lvert \btheta\rvert\,\bb(1)$.
	\end{itemize}
	Define $L_r\coloneqq\dim(\bV_r)$ so that $\bV_r = (\bV_r(\balpha):\balpha \in \Pi(L_r))$.
	Denote by $\P_{\bv}$ the probability measure associated with $\bV$ under the initial condition $\bV_0=\bv$, 
	and by $\E_{\bv}$ the associated expectation. 
\end{definition}

The next result is the Bernstein duality between the $\Lambda$-Wright--Fisher process and the Bernstein coefficient process.
The proof can be found in Section~\ref{sec:proofDuality}.

\begin{theorem}[Bernstein duality]\label{thm:Bernsteinduality}
    Fix parameters~$\cP$.
    For every $\bx\in\Delta_d$ the martingale problem for $(\cA,C^2(\Delta_d))$ has a solution $\bX$ with $\bX_0=\bx$,
    and every such solution satisfies, for all $\bv\in\bM$ and $t\geq 0$,
    \[
    \E\left[\sum_{\balpha\in \Pi(\dim(\bv))} \bv(\balpha)\binom{\dim(\bv)}{\balpha}\bX_t^{\balpha}\right]=\E_{\bv}\left[\sum_{\balpha\in \Pi(L_t)} \bV_t(\balpha)\binom{L_t}{\balpha}\bx^{\balpha}\right].
    \]
    Consequently, the martingale problem for $(\cA,C^2(\Delta_d))$ is well posed, 
    and its solution is a Feller process whose generator acts on $C^2(\Delta_d)$ as~$\cA$.
\end{theorem}

The underlying \textit{duality function} is thus $H:\Delta_d\times \bM\to\R$ defined via 
\begin{align}
	H(\bx,\bv)\coloneqq \sum_{\balpha\in \Pi(\dim(\bv))} \bv(\balpha)\binom{\dim(\bv)}{\balpha}\bx^{\balpha} . \label{defh}
\end{align}

This completes the first aim of the present manuscript, to extend the Bernstein duality in~\cite{Cordero2022} to the multi-dimensional case.

Since $L_r=\dim(\bV_r)$ is the degree of the polynomial $\bx\mapsto H(\bx,\bV_r)$, 
we call $L$ the \emph{degree process}.

\subsection{Long-term behaviour}

For the long-term behaviour results, we impose the following condition on the parameters:

\begin{equation}\label{assumption:recurrence_L}
	\sum_{m=2}^{\kappa} (m-1)\bb(m) <\int_{[0,1]}\lvert \log(1-z)\rvert \frac{\Lambda(\dd z)}{z^2}+\lvert\btheta\rvert\, \bb(1).
\end{equation}
This assumption is for example satisfied if $\Lambda(\{0\})>0$; so all our results cover Wright--Fisher diffusions with the specified drift term.
In particular, from now on~$\cP$ denotes a generic choice for the model parameters that satisfy~\eqref{assumption:recurrence_L}.

The relevance of condition~\eqref{assumption:recurrence_L} is that when it holds, the Bernstein coefficient process turns out to be well behaved in the sense that the long-term behaviour of the corresponding degree process is known.
We recall the corresponding result rephrased in our notation; it is essentially Lemma~4.8 in~\cite{Cordero_2026}, and we refer to that reference for the proof 
(which in turn is strongly inspired by the method developed in~\cite{foucart2013impact}; however, see also~\cite{GonzalezCasanova2020} for a different approach).

\begin{lemma}[Recurrence and absorption of~$L$] \label{lem:L_rec_abs}
    Fix parameters~$\cP$. 
    $L$ is conservative.
	If~\eqref{assumption:recurrence_L} is satisfied, the following assertions hold.
	\begin{enumerate}
		\item If $\lvert \btheta\rvert\bb(1)=0$, then $1$ is positive recurrent for $L$. 
		\item If $\lvert \btheta\rvert \bb(1)>0$, then $L$ is absorbed at $0$ in finite time. 
	\end{enumerate}
\end{lemma}
The long-term behaviour of the degree process is useful for determining the long-term behaviour of~$\bV$ itself.
We distinguish between parameter sets having one of the following properties
\begin{align}
    &\bb(1)=0. \label{cond:absorption}\\
    &\lvert \btheta\rvert\bb(1)>0 \text{ and } P\text{ is irreducible} \label{cond:stationary_parent_independent}, \\
    &\bb(1)>0,\ \lvert \btheta \rvert =0, \kappa=1,\ P\text{ is irreducible and aperiodic}.\label{cond:stationary_true_parent_dependent}
\end{align}
From a biological perspective, condition \eqref{cond:absorption} corresponds to the absence of mutation, 
while \eqref{cond:stationary_parent_independent} describes a mutation process with a parent-independent component. 
Finally, \eqref{cond:stationary_true_parent_dependent} provides an instance of a strictly parent-dependent mutation mechanism.

The next two results deal with~\eqref{cond:absorption}.
The first result describes under~\eqref{assumption:recurrence_L} the positive recurrence of the Bernstein coefficient process; 
whereas the second result exhibits implications of that positive recurrence for the long-term behaviour of~$\bX$.
The corresponding proofs are presented in Section~\ref{sec:prooffixation}.

\begin{proposition}[Long-term behaviour of $\bV$ if $\bb(1)=0$]\label{prop:asympV}
	Fix parameters~$\cP$ such that $\bb(1)=0$.
    Then, $(\bV_{r}(L_r\be{i}))_{r\geq 0}$ is constant for all $i\in\llbracket d+1 \rrbracket $. 
	Moreover, if in addition~\eqref{assumption:recurrence_L} is satisfied, then the following holds for every $\hat{\bv}\in \bM(1)$.
	\begin{enumerate}  
		\item The Bernstein coefficient process $\bV$ on $\cS_{\cP}(\hat{\bv})$ has a unique invariant probability measure $\mu^{\hat{\bv}}$ with 
		support $\cS_{\cP}(\hat{\bv}) \subset \{\bw\in \bM : \  \forall i \in\llbracket d+1 \rrbracket   \ \bw(\dim(\bw)\be{i})=\hat{\bv}(\be{i})\}$. 
		\item Let $\bV_{\infty}^{\hat{\bv}}$ be a random variable with law~$\mu^{\hat{\bv}}$.
        If $\bV_0=\bw \in \bM$ with $\bw(\dim(\bw)\be{i})=\hat{\bv}(\be{i})$ for all $i\in\llbracket d+1 \rrbracket $, 
        then $\cS_{\cP}(\hat{\bv}) \subset \cS_{\cP}(\bw)$ and $\bV_r$ converges in distribution to $\bV_{\infty}^{\hat{\bv}}$ as $r\to\infty$.
	\end{enumerate}
\end{proposition}

The long-term behaviour of~$\bV$ can be used to deduce the long-term behaviour of~$\bX$ via the Bernstein duality.
This is the content of the next result.
Recall from the notation section, that for $i\in \llbracket d+1 \rrbracket$, the function $\obv(i)\in \bM(1)$ assigns~$1$ to $\be{i}$ and~$0$ else.

\begin{theorem}[Absorption of $X$ if $\bb(1)=0$]\label{fixext}
	Fix parameters~$\cP$ such that~\eqref{assumption:recurrence_L} holds and $\bb(1)=0$.
	Let $\bX_0=\bx$ for some $\bx\in \mathrm{int}(\Delta_d)$.
    Then $\bX_t$ converges almost surely to a random variable $\bX_\infty$ taking values only in $\{\be{i}\in \Delta_d: i\in \llbracket d+1\rrbracket \}$, and each such value is taken with positive probability.
	Moreover,
	\begin{align}
		\P_\bx(\bX_\infty=\be{i})=\E[H(\bx,\bV^{\obv(i)}_{\infty})]\in(0,1), \label{exprfivproba}
	\end{align}
	where the random variable $\bV^{\obv(i)}_{\infty}$ is as in Proposition \ref{prop:asympV}. 
\end{theorem}

Proposition~\ref{prop:asympV} and Theorem~\ref{fixext} address our objective of delineating the conditions for absorption and offering a representation of the absorption probability (so aim~$2$ and~$3$ outlined in the introduction).

Next, we consider the long-term behaviour under~\eqref{cond:stationary_parent_independent} or~\eqref{cond:stationary_true_parent_dependent}.
The next two results are proved in Section~\ref{sec:proofstationary}.

\begin{proposition}[Long-term behaviour of $\bV$ under~\eqref{cond:stationary_parent_independent} or~\eqref{cond:stationary_true_parent_dependent}] \label{prop:absV}
	Assume \eqref{assumption:recurrence_L} holds.
    If~\eqref{cond:stationary_parent_independent} or~\eqref{cond:stationary_true_parent_dependent} holds, for any $\bv\in \bM$, $\P_{\bv}$-almost surely $\lim_{r\to\infty}\min_{\balpha\in\Pi(L_r)}\bV_r(\balpha)=\lim_{r\to\infty}\max_{\balpha\in\Pi(L_r)}\bV_r(\balpha)= U_{\infty}$ for some random real-valued $U_{\infty}$. 
    Under~\eqref{cond:stationary_parent_independent},
    $\bV$ is absorbed in finite time in $\bV_\infty\in \cS_{\cP}(\bv) \cap \bM(0)$ and $\bV_{\infty}(0,\ldots,0)=U_{\infty}$.
\end{proposition}
Under~\eqref{cond:stationary_true_parent_dependent}, $U_{\infty}$ can be explicitly computed using the stationary distribution of the primitive stochastic matrix~$P$ and the value of~$\bV$ when entering~$\bM(1)$ for the first time; 
we refer to the proof for details.
The following theorem shows that under~\eqref{cond:stationary_parent_independent} or~\eqref{cond:stationary_true_parent_dependent} 
the process $\bX$ has a unique stationary distribution, 
which can be expressed in terms of~$U_{\infty}$.
To state it, for $n\in\N_0$ and $\balpha\in\Pi(n)$ let $\bv^{\balpha}\in\bM(n)$ be given by
$\bv^{\balpha}(\balpha)=\binom{n}{\balpha}^{-1}$ and $\bv^{\balpha}(\bbeta)=0$ for
$\bbeta\in\Pi(n)\setminus\{\balpha\}$, and write $\P_{\balpha}\coloneqq\P_{\bv^{\balpha}}$
and $\E_{\balpha}\coloneqq\E_{\bv^{\balpha}}$.

\begin{theorem}[Ergodicity of $\bX$ under~\eqref{cond:stationary_parent_independent} or~\eqref{cond:stationary_true_parent_dependent}]\label{thstadist}
	Assume \eqref{assumption:recurrence_L} and either~\eqref{cond:stationary_parent_independent} or~\eqref{cond:stationary_true_parent_dependent}.
    Then, for any $\bx \in \Delta_d$ and $i\in \llbracket d+1 \rrbracket$, $\P_{\bx}(\bX_t \xrightarrow[]{t\to\infty} \be{i})=0$ and $\bX$ has a unique stationary distribution $\mathcal{L}$ that is characterised by its (mixed) moments
	\begin{align}
		\forall {\balpha} \in \cup_{n\geq 0} \Pi(n), \quad \int_{\Delta_d} \bx^{\balpha} \mathcal{L}(\dd \bx)=\E_{\balpha}[U_\infty], \label{momentsloiinvar}
	\end{align}
	where $U_{\infty}$ is as in Proposition~\ref{prop:absV}.
	Moreover, for any initial distribution on $\Delta_d$, the law of $\bX_t$ converges weakly to $\mathcal{L}$ as $t\to\infty$. 
\end{theorem}

Proposition~\ref{prop:absV} and Theorem~\ref{thstadist} address our goal of describing the conditions required for the existence of a non-trivial ergodic measure for~$\bX$, as well as providing a representation of that measure (so aim~$2$ and~$4$ outlined in the introduction).

\subsection{Simulation approach for moments of $\Lambda$-Wright--Fisher processes}\label{sec:example}

In this section, we lay out a procedure to use Bernstein duality and the dual process to simulate the moment process of the $\Lambda$-Wright--Fisher process. 
First, let us define the moment process: for $\balpha\in \Pi(n)$ and $\bx\in \Delta_d$, 
we define $u_{\balpha,\bx}(t)\coloneqq \E_{\bx}[\bX_t^{\balpha}]$ for $t\geq 0$. 
We call $u=(u_{\balpha, \bx}(t):t\geq 0)$ the $\balpha$-moment process of $\bX$ started from $\bx$.

For the simulations, we exploit the duality.
Fix $\balpha\in\Pi(n)$.
Theorem~\ref{thm:Bernsteinduality} then yields
\[
u_{\balpha,\bx}(t)=\E_{\bx}\big[H(\bX_t,\bv^{\balpha})\big]
=\E_{\bv^{\balpha}}\big[H(\bx,\bV_t)\big],\qquad t\ge 0,\ \bx\in\Delta_d .
\]

The strategy for simulating the moment process of $\bX$ is straightforward: 
we compute the empirical average of $(H(\bx, \bV_t) : t \geq 0)$ over multiple independent realisations of the dual process.
What makes the Bernstein coefficient process particularly tractable for simulation is that, even though the forward process is a jump-diffusion, the dual process is a pure jump process.

Fix the mixed moment $\balpha$ to be computed, and a time horizon $T > 0$. 
To obtain~$N$ realisations of the process, we perform the following steps for each $j \in \{1, \dots, N\}$:
First, we simulate the degree process $L^{(j)} = (L_t^{(j)} : t \in [0,T])$ starting from $n=\lvert \balpha\rvert$.
Let $0 = t_0 < t_1 < \dots < t_K$ denote the jump times of the process (including null transitions).
We initialise $\bV_0^{(j)}=\bv^{\balpha}$.
Second, conditioned on the realisation of $L^{(j)}$, let $M_i$ be the Bernstein operator corresponding to the $i$-th transition of $L^{(j)}$ for $i > 0$. 
For each jump time $t_i$, we update the state via $\bV_{t_i}^{(j)} = M_i \bV_{t_{i-1}}^{(j)}$. 
Since $\bV$ remains constant between transitions, this yields a full realisation $(\bV_t^{(j)} : t \in [0,T])$, from which we compute the sample path $u_{\balpha,\bx}^{(j)}(t) = H(\bx, \bV_t^{(j)})$ for all $t \in [0,T]$.
Finally, we approximate the moment process by the empirical mean:
\[
\hat{u}_{\balpha,\bx}(t) = \frac{1}{N} \sum_{j=1}^N u_{\balpha,\bx}^{(j)}(t).
\]

The most computationally costly component of this procedure is the multiplication of the Bernstein operators. 
If we are only interested in the expectations $(\E_{\balpha}[H(\bx,\bV_T)]:\bx\in\Delta_d)$, we can, 
for each realisation~$j$, 
concatenate the operators into a single composite Bernstein matrix $M^{(j)}\coloneqq M_K\cdots M_1$,
which maps $\bV_0^{(j)}$ to $\bV_T^{(j)}$. 
Applied to several initial vectors, e.g.\ to all $\bv^{\balpha}$ with $\balpha\in\Pi(n)$, it yields estimates of all moments of order~$n$ from the same realisation of~$L$.
Moreover, since $H$ is linear in its second argument, the final coefficient vectors $\bV_T^{(j)}$ can be averaged over all realisations, grouped by their degree $L_T^{(j)}$.
The estimate is then a polynomial in $\bx$ that can be evaluated at all initial conditions of interest at once. 
Absorption probabilities can be estimated in the same spirit without a time horizon using regenerative properties of~$\bV$.

Let us note in passing that 
when simulating the $\Lambda$-Wright--Fisher moment process, 
one can moreover exploit a coupling of the dual processes started from different initial conditions.
This allows for the simultaneous construction of trajectories for all moments of order $n' \leq n$.
Such a coupling is based on a corresponding coupling of the degree processes (see, e.g., \cite[Rem.~5.8]{Cordero2022}).

\subsubsection{Simulation of $\Lambda$-Wright--Fisher process with nontransitive fitness}

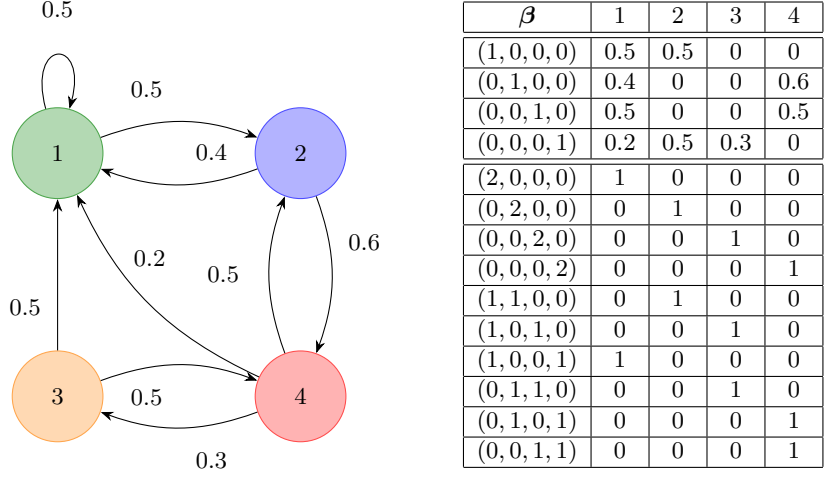
\begin{figure}[t]
    \centering
    \begin{minipage}[c]{0.5\textwidth}
        \begin{tikzpicture}[>=Stealth,
                node distance=2cm,
                every node/.style={circle,draw,minimum size=1.2cm,inner sep=0pt},
                every edge/.style={->,thick,shorten >=2pt,shorten <=2pt}
            ]
            
            \node[fill=Green!30, draw=Green!70, text=black]  (1) {1};
            \node[fill=blue!30, draw=blue!70, text=black]  (2) [right=of 1] {2};
            \node[fill=orange!30, draw=orange!70, text=black]  (3) [below=of 1] {3};
            \node[fill=red!30, draw=red!70, text=black]  (4) [right=of 3] {4};
            
            \draw[->,loop above] (1) to node[midway,above, fill=none, draw=none] {0.5} (1);
            \draw[->,bend left=20] (1) to node[midway, above left, fill=none, draw=none] {0.5} (2);
            
            \draw[->,bend left=20] (2) to node[midway, above right, fill=none, draw=none] {0.4} (1);
            \draw[->,bend left=20] (2) to node[midway, above right, fill=none,draw=none] {0.6} (4);
            
            \draw[->] (3) to node[midway, below left, fill=none,draw=none] {0.5} (1);
            \draw[->, bend left = 20] (3) to node[midway, below left, fill=none,draw=none] {0.5} (4);
            
            \draw[->,bend left=20] (4) to node[midway, above , fill=none,draw=none] {0.2} (1);
            \draw[->,bend left=20] (4) to node[midway, left, fill=none,draw=none] {0.5} (2);
            \draw[->,bend left=20] (4) to node[midway, below right, fill=none,draw=none] {0.3} (3);
            
            \end{tikzpicture}
    \end{minipage}~~~\begin{minipage}[t]{0.5\textwidth}
        \begin{tabular}{|*{5}{c|}}
            \hline
            $\bbeta$    & $1$  & $2$ & $3$ & $4$  \\ \hline \hline
            $(1,0,0,0)$ & $0.5$ & $0.5$ & $0$ & $0$ \\ \hline
            $(0,1,0,0)$ & $0.4$ & $0$ & $0$ & $0.6$ \\ \hline
            $(0,0,1,0)$ & $0.5$ & $0$ & $0$ & $0.5$ \\ \hline
            $(0,0,0,1)$ & $0.2$ & $0.5$ & $0.3$ & $0$ \\ \hline \hline
            $(2,0,0,0)$ & $1$ & $0$ & $0$ & $0$ \\ \hline
            $(0,2,0,0)$ & $0$ & $1$ & $0$ & $0$ \\ \hline
            $(0,0,2,0)$ & $0$ & $0$ & $1$ & $0$ \\ \hline
            $(0,0,0,2)$ & $0$ & $0$ & $0$ & $1$ \\ \hline
            $(1,1,0,0)$ & $0$ & $1$ & $0$ & $0$ \\ \hline
            $(1,0,1,0)$ & $0$ & $0$ & $1$ & $0$ \\ \hline
            $(1,0,0,1)$ & $1$ & $0$ & $0$ & $0$ \\ \hline
            $(0,1,1,0)$ & $0$ & $0$ & $1$ & $0$ \\ \hline
            $(0,1,0,1)$ & $0$ & $0$ & $0$ & $1$ \\ \hline
            $(0,0,1,1)$ & $0$ & $0$ & $0$ & $1$ \\ \hline
        \end{tabular}
    \end{minipage}
    \caption{Parameters of the $\Lambda$-Wright--Fisher process with parent-dependent mutation kernel and nontransitive fitness as in the example in Section~\ref{sec:example}.
    Left: graph of the underlying parent-dependent mutation matrix.
    Right: $\bp(\bbeta;i)$ for each viable $\bbeta$ and~$i\in \llbracket 4\rrbracket$. }
    \label{fig:example_model}
\end{figure}

\begin{figure}[t]
    \centering
    \includegraphics[width=\linewidth]{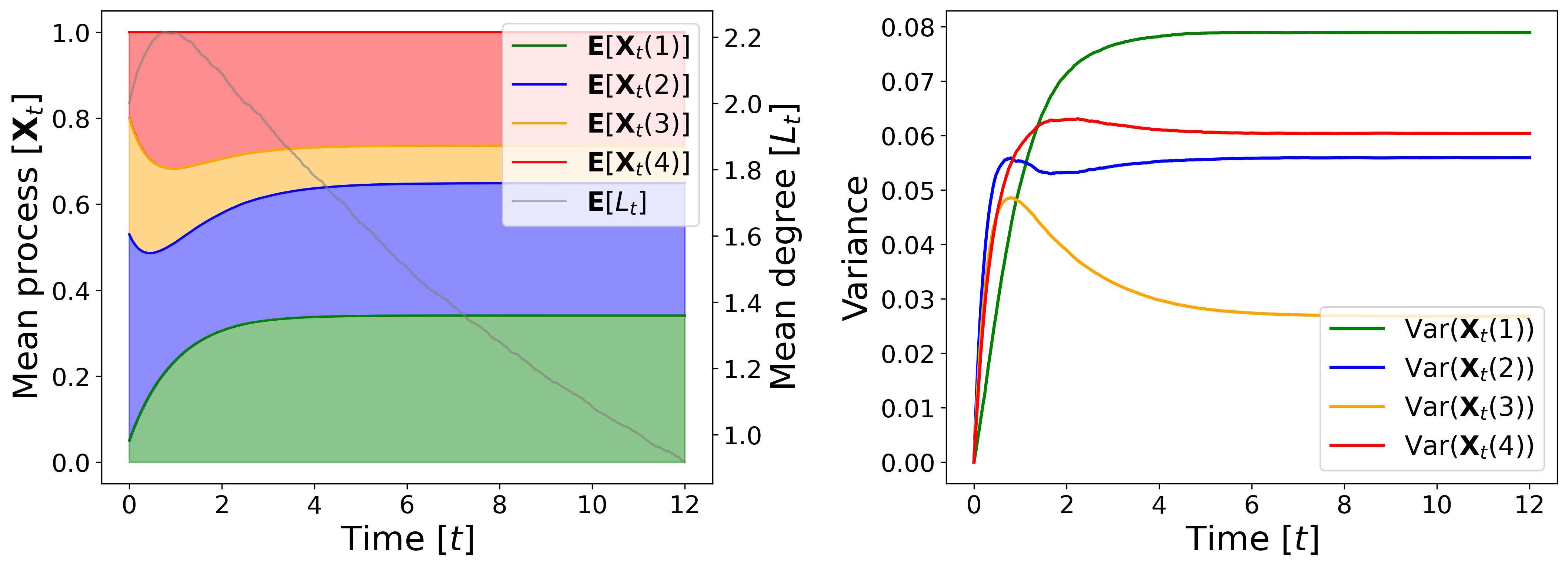}
    \caption{Case~\eqref{cond:stationary_parent_independent}, i.e.~$\bb(1)=1$, and~$\bx_0=[0.05, 0.48, 0.27, 0.2]$. In both plots, $\bX(1)$ (green), $\bX(2)$ (blue), $\bX(3)$ (orange), $\bX(4)$ (red) based on $N=75{,}000$ realisations and for $t\in [0,12]$.
    Left: Mean process (left axis) and superimposed the mean degree process (grey, right axis, started from $L_0=2$).
    Right: Variance process.}
    \label{fig:plot_b1_positive}
    
\end{figure}
To illustrate this simulation ansatz, we consider a special case with $d+1=4$ types. 
We fix $\Lambda(dx)=dx$, corresponding to the Bolthausen--Sznitman case, and set $\kappa = 2$.
In particular, the integral on the right-hand side of~\eqref{assumption:recurrence_L} diverges and thus the corresponding assumption is satisfied.
We focus on a scenario of nontransitive fitness (akin to a rock-paper-scissors dynamic), 
where type $2$ is strictly fitter than type $1$, type $3$ is strictly fitter than types $1$ and $2$, type $4$ is strictly fitter than types $2$ and $3$, and type $1$ is strictly fitter than type $4$. 

This dynamic is formalised by choosing $\bbeta$ with $\lvert \bbeta\rvert = 2$, as listed in the lower part of the table in Figure~\ref{fig:example_model}. 
For simplicity, we assume $\bb(2)=1$. 
We will explore two scenarios: $\bb(1)=0$ (no mutation) and $\bb(1)=1$ (mutation at rate 1), the latter with the mutation matrix listed in the upper part of the table in Figure~\ref{fig:example_model}. 
All drift parameters are summarised in that same figure.

For the model with mutation ($\bb(1)=1$), which possesses a stationary distribution, we set the initial condition for concreteness to $\bx_0=[0.05, 0.48, 0.27, 0.2]$.
We computed the first moment processes, $(\E_{\bx_0}[\bX_t(i)] : t \in [0,T])$ for $i \in \{1, \dots, d+1\}$, as well as the variance processes of type~$i$, defined as $(\E_{\bx_0}[\bX_t(i)^2] - \E_{\bx_0}[\bX_t(i)]^2 : t \in [0,T])$.
We used $N=75{,}000$ independent realisations.
Since $H(\bx,\bV_t)\in[0,1]$, 
Hoeffding's inequality implies that, 
for each fixed $t$ and each estimated moment, 
the empirical mean deviates from the true value by more than $0.005$ with probability at most~$5\%$ (as $2\exp(-2N\cdot 0.005^2)<0.05$).
Consequently, each estimated variance deviates from the
true variance by at most $0.015$ with probability at least~$90\%$.
The variance process is of particular interest for parameter inference.
The results of these simulations for the case with mutation are presented in Figure~\ref{fig:plot_b1_positive}.

For the case without mutation ($\bb(1)=0$), 
we demonstrate the simulation approach by estimating the deviation of the absorption probability from its neutral value: 
without drift, $\bX(i)$ is a bounded martingale, so that $\P_{\bx}(\bX_\infty=\be{i})=\bx(i)$.
Lemma~\ref{lem:boundary_i_nomutation} (with~\eqref{assumption:recurrence_L} in place) implies that for every $i\in \llbracket d+1\rrbracket$, under $\P_{\obv(i)}$,
we have that $\bV=\obv(i)$ whenever $L$ returns to~$1$.
Hence $\bV$ is regenerative at its returns to $\obv(i)$, 
and by (the proof of) Proposition~\ref{prop:asympV} the cycle length $\tau\coloneqq\inf\{t>J_1:\bV_t=\obv(i)\}$, 
where $J_1$ is the first jump time of~$\bV$,
is integrable.
Together with Lemma~\ref{lemf} and the cycle formula for the stationary law of regenerative processes \cite[Ch.~VI, Thm. 1.2]{asmussen2008applied}, this yields
\[
\P_{\bx}(\bX_\infty=\be{i})
=\frac{\E_{\obv(i)}\big[\int_0^{\tau}H(\bx,\bV_s)\,\dd s\big]}{\E_{\obv(i)}[\tau]}.
\]
We estimate the right-hand side by the ratio of empirical means over $75{,}000$ independent cycles. 
The result for $i=2$ (and for $\bx$ with $\bx(3)=0.2$) is displayed in Figure~\ref{fig:plot_b1_0}.

\begin{figure}[t]
        \centering
        \includegraphics[width=0.6\linewidth]{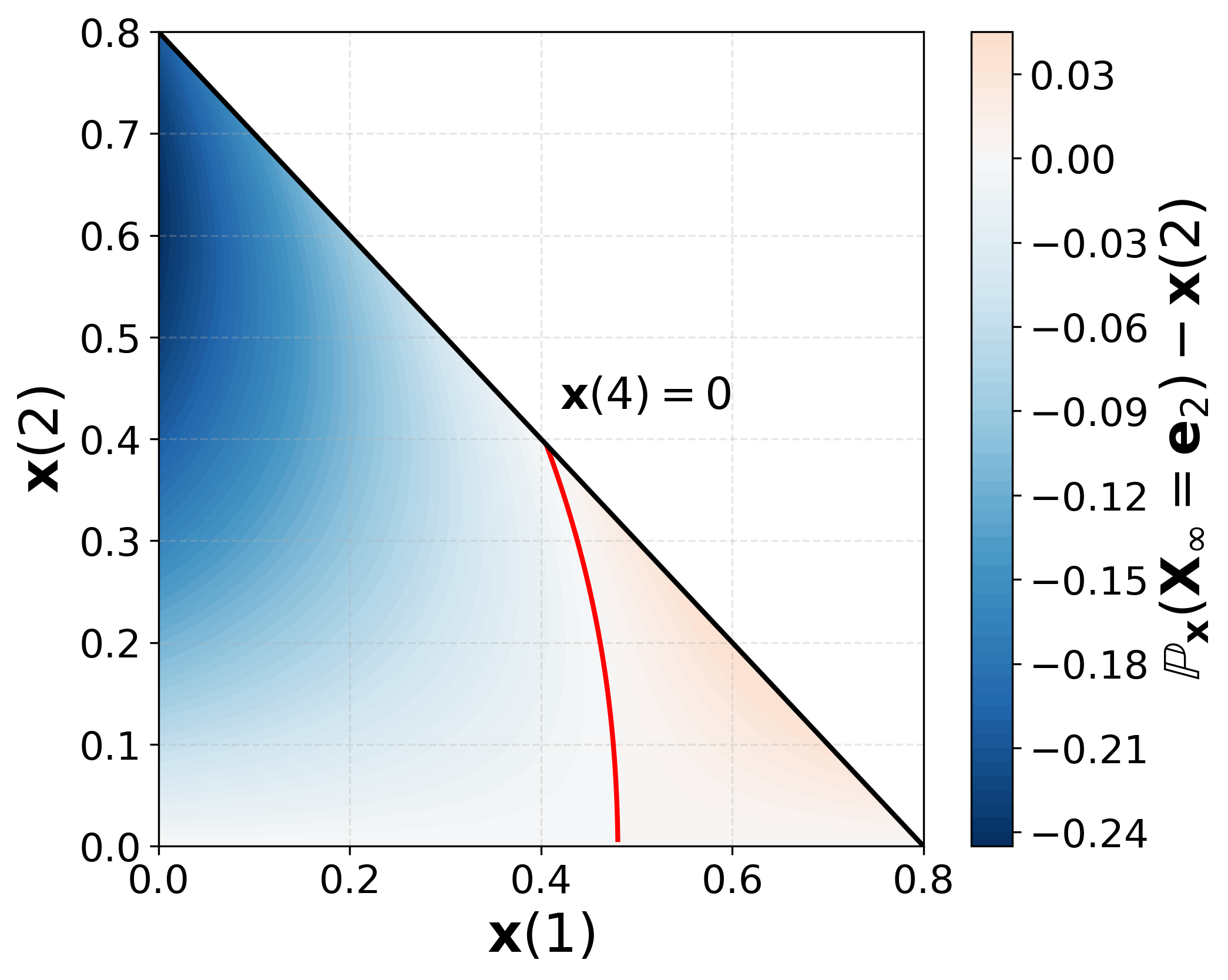}
        \caption{Case~\eqref{cond:absorption}, 
        i.e.\ $\bb(1)=0$. 
        Estimate of the deviation $\P_{\bx}(\bX_\infty=\be{2})-\bx(2)$ of the absorption (fixation) probability of type~$2$ from its neutral value, for $\bx\in\Delta_d$ with $\bx(3)=0.2$, obtained with the regenerative estimator from $75{,}000$ independent cycles.
        The axes show $\bx(1)$ and $\bx(2)$, with the hypotenuse corresponding to $\bx(4)=0$.
        The red curve is the zero level set, separating initial conditions where selection increases (red) or decreases (blue) the fixation probability of type~$2$ relative to the neutral (non-selective) case.}
    \label{fig:plot_b1_0}
\end{figure}

A direct Monte Carlo approach would simulate $\bX$ itself, e.g.\ by an Euler scheme.
This requires a separate set of simulations for every initial condition $\bx$, introduces a time-discretisation bias, and requires truncating the small jumps, since $r^{-2}\Lambda(\dd r)$ has infinite mass for $\Lambda(\dd r)=\dd r$.
In contrast, the dual process is a pure jump process that can be simulated exactly,
and a single set of cycles yields the estimate for all $\bx$ simultaneously.
The moment equations of $\bX$ do not depend on $\bx$ either, but for $\kappa\ge 2$
they do not close, since the degree of the moments increases.
So solving the system of moments numerically requires a truncation, which the simulation approach avoids.

We note some limitations of the current approach.
Under the selected parameters, the degree process generally remains small and thus the Bernstein coefficient process remains computationally tractable. 
In settings where the degree process reaches larger values for longer durations, the performance of the Bernstein-based approach degrades markedly, partly because the cardinality of the set $\Pi(n)$ grows as $\binom{n+d}{d} \sim n^{d}/d!$.

The source code for the simulations will be made available on a public repository.

\subsection{Open problems}
Several open problems remain. 
We discuss a few briefly.

 \begin{enumerate}
     \item \textit{Characterisation of drifts:}
    For Wright--Fisher processes with drifts of the form~\eqref{eq:driftterm}, how large is the subset of drifts that admit a representation as in~\eqref{eq:drift_form}? What about other drifts?
    As a start one could compare in the diffusion case ($\Lambda=\delta_0$) with \cite{shiga1981}, where for that case all viable drifts are characterised.
    Another possible approach would be to generalise the Bernoulli-factory method from the two-type setting \cite{Koskela2024, koskela_well_posedness_2026} to the multi-type case.

    \item \textit{Long-term behaviour for general Markov matrices ~$P$}: 
    When $\bb(1)>0$, we treated irreducible $P$ for the long-term behaviour of the Wright--Fisher process. 
    A more careful analysis should provide a unified treatment for arbitrary Markov matrices $P$, but the implementation might be tedious.

    \item \textit{Selection with purely parent-dependent mutation}:
    Our results do not cover the case $\bb(1)>0$, $\lvert\btheta\rvert=0$ and $\kappa\geq 2$, in which $L$ is positive recurrent but not absorbed.
    Is $\bX$ ergodic in this case, and can its stationary distribution be represented in terms of~$\bV$?

    \item \textit{Algorithmic efficiency}: 
    From an applied perspective, is there a more efficient implementation of the algorithm for the moment process associated with our construction? 
    In particular, can the degree process be controlled or truncated in some way to mitigate the computational cost when it remains large for extended periods?
    
    \item \textit{Transition function expansion}: 
    Can Bernstein duality be used for the derivation of an explicit or approximate transition function expansion?
    \cite{barbour2000} derived in the diffusion case ($\Lambda=\delta_0$) such an expansion using moment duality.
    Can their approach be translated to the Bernstein setting?

    \item \textit{Mixing times}: Moment duality has been used to derive strong stationary times for neutral $\Lambda$-Wright--Fisher processes with mutation \cite{Blancas2023}. 
    Can the Bernstein duality framework be used to derive strong stationary times for more general ergodic $\Lambda$-Wright--Fisher processes? 
    If strong stationary times prove too difficult to obtain, can the mixing rate be established using the recently developed methods of \cite{cordero_distance_2025} in the multi-type setting?
 \end{enumerate}

\section{Proof of the Bernstein duality}\label{sec:proofDuality}

The proof of the Bernstein duality combines an algebraic duality of the generators with the martingale-problem approach of~\cite[Ch.~4]{Ethier1986}.
First, we state the algebraic duality of the generators.
The proof is a lengthy (yet straightforward) calculation and is therefore only given in Appendix~\ref{sec:generator_identity}.
\begin{lemma}\label{lem:generator_identity}
    Let $\cB$ be the generator of the Bernstein coefficient process. 
    For $\bx\in \Delta_d$ and~$\bv\in \bM$, 
    \begin{equation}
		\cA H(\cdot,\bv)(\bx)=\cB H(\bx,\cdot)(\bv).\label{eq:generatorrelation}
	\end{equation}
\end{lemma}

The following lemma provides some basic properties of the operators introduced in Definition~\ref{def:bramutoperators}.

\begin{lemma}\label{lem:actionV}
	Fix the process parameters~$\cP$.
    Let $n\in \N$ and $\bv\in \bM(n)$. 
    Then for any $k\in \llbracket 2,n\rrbracket$, $m\in\llbracket 0,\kappa\rrbracket$, we have $\lVert C^{n,k} (\bv)\rVert_{\infty}\leq \lVert \bv \rVert_{\infty}$ and $\lVert D^{n,m} (\bv)\rVert_{\infty} \ \leq \lVert \bv \rVert_{\infty}$. In particular, for all $\bw\in \cS_{\cP}(\bv)$: $\lVert\bw\rVert_{\infty}\leq\lVert\bv\rVert_{\infty}$.
\end{lemma}
\begin{proof}
	The result follows in a straightforward way from the definition of the operators.
\end{proof}

The following lemma provides some bounds on~$\bV$ needed in the proof of Theorem~\ref{thm:Bernsteinduality}.

\begin{lemma}\label{condq}
Fix the parameter set~$\cP$. 
Let~$\cB$ be the rate matrix of~$\bV$, and set $c\coloneqq\tfrac12\Lambda([0,1])+\sum_{m=1}^{\kappa}\bb(m)$. 
Then for all $\bw\in\bM$,
\begin{equation}\label{conditionq2}
\sum_{\bu\in\bM \setminus\{\bw\}}\cB(\bw,\bu)\leq c\dim(\bw)^2 .
\end{equation}
Moreover, $\bV$ is non-explosive, 
and $\E_{\bv}[\sup_{s\leq t}L_s^2]<\infty$ for all $\bv\in\bM$ and $t\geq0$.
\end{lemma}

\begin{proof}
Let $\bw\in\bM(n)$. 
The transitions $C^{n,k}$, $k\in\llbracket2,n\rrbracket$, 
have total rate $\sum_{k=2}^n\tbinom{n}{k}\lambda_{n,k}
\leq \tbinom{n}{2}\Lambda([0,1])$.
The transitions $D^{n,0}$ and $D^{n,1}$ have total rate $n\lvert\btheta\rvert\bb(1)+n(1-\lvert\btheta\rvert)\bb(1)=n\,\bb(1)$, 
and
$D^{n,m}$ has rate $n\,\bb(m)$ for $m\in\llbracket2,\kappa\rrbracket$.
Summing gives~\eqref{conditionq2}.

Next, $L$ can be coupled with a pure-birth process~$Y$, started at $Y_0=L_0$ and jumping from $n$ to $n+m-1$ at rate $n\,\bb(m)$ for $m\in\llbracket2,\kappa\rrbracket$, 
such that $L_s\leq Y_s$ for all $s\geq0$. 
Indeed, while $L_s\leq Y_s$, the upward rates satisfy $L_s\bb(m)\leq Y_s\bb(m)$, 
so every upward jump of~$L$ can be realised as a jump of the same size of~$Y$ by thinning, 
and all other jumps of~$L$ are downward.
Since $Y$ is a non-decreasing continuous-time Markov branching process with bounded
offspring, 
$Y_t<\infty$ almost surely and $\E[Y_t^2]<\infty$ for all $t\geq0$.
Hence $\E_{\bv}[\sup_{s\leq t}L_s^2]\leq\E[Y_t^2]<\infty$. 
Moreover, by~\eqref{conditionq2}
the jump rates of $\bV$ on $[0,t]$ are bounded by $cY_t^2<\infty$, 
so $\bV$ makes only finitely many jumps on bounded time intervals, 
i.e.\ $\bV$ is non-explosive.
\end{proof}

Finally, we prove the well-posedness of the martingale problem associated to $\cA$, and the Bernstein duality.

\begin{proof}[Proof of Theorem~\ref{thm:Bernsteinduality}]
First we establish existence of a solution to the $(\cA,C^2(\Delta_d))$-martingale problem and the duality.
The operator $\cA$ maps $C^2(\Delta_d)$ into $C(\Delta_d)$ and satisfies the positive
maximum principle.
Thus, by~\cite[Thm.~4.5.4]{Ethier1986}, a solution to the martingale problem exists for every initial condition.
Now let $\bX$ be any solution with $\bX_0=\bx$. Fix $\bv\in\bM$ and let $\bV$ be
independent of~$\bX$ with $\bV_0=\bv$.
By Lemma~\ref{lem:generator_identity}, $\cA H(\cdot,\bw)(\bx)=\cB H(\bx,\cdot)(\bw)$ for all $\bx\in\Delta_d$ and $\bw\in\bM$.
Using Lemma \ref{lem:actionV}, 
we obtain
	\begin{equation}\label{int1}
	\sup_{\bx\in \Delta_d ,\bw \in \cS_{\cP}(\bv)} \lvert H(\bx,\bw)\rvert \leq \lVert \bv\rVert_\infty, 
	\end{equation}
so $H(\cdot,\cdot)$ is bounded on $\Delta_d\times \cS_{\cP}(\bv)$.
By Lemma~\ref{condq}, $\bV$ is non-explosive, and by~\eqref{int1} and~\eqref{conditionq2},
$\lvert\cB H(\bx,\cdot)(\bw)\rvert\le 2c\lVert\bv\rVert_\infty\dim(\bw)^2$ for all
$\bx\in\Delta_d$ and $\bw\in\cS_{\cP}(\bv)$, with $c$ as in Lemma~\ref{condq}.
Moreover, $\E_{\bv}[\sup_{s\le t}L_s^2]<\infty$ by Lemma~\ref{condq}.
Hence the integrability conditions of \cite[Thm.~4.4.12]{Ethier1986} are satisfied.
In particular, \cite[Cor.~4.4.13]{Ethier1986} yields, for all $t\geq 0$, $\E[H(\bX_t,\bv)]=\E_{\bv}[H(\bx,\bV_t)]$.
The same argument applies if $\bX_0\sim\nu$ for a probability measure $\nu$ on $\Delta_d$ with $\bX_0$ independent of $\bV$.

Now, we establish well-posedness.
The functions $H(\cdot,\bv)$, $\bv\in\bM$, span all polynomials, which are separating on
$\Delta_d$. 
By the first part of this proof, 
the one-dimensional distributions of any solution are therefore uniquely determined, 
and uniqueness of the martingale problem follows from~\cite[Thm.~4.4.2]{Ethier1986}. 
Together with existence, the martingale problem is well posed.

Next, we prove that $\bX$ is Feller.
Let $T_tf(\bx)\coloneqq\E_{\bx}[f(\bX_t)]$ for $f\in C(\Delta_d)$. 
For $f=H(\cdot,\bv)$, the first part of this proof gives
$T_tf(\bx)=\E_{\bv}[H(\bx,\bV_t)]$, 
which is
continuous in~$\bx$ by dominated convergence, 
since each $H(\cdot,\bw)$ is a polynomial bounded
by $\lVert\bv\rVert_\infty$ because of~\eqref{int1}. 
Moreover, $H(\bx,\bV_t)=f(\bx)$ unless $\bV$ jumps
before time~$t$.
Since $\bV$ has a finite jump rate at~$\bv$ by Lemma~\ref{condq} we thus have
$\sup_{\bx}\lvert T_tf(\bx)-f(\bx)\rvert\le
2\lVert\bv\rVert_\infty\P_{\bv}(\bV \text{ jumps before } t)\to0$ as $t\to0$.
Since every polynomial on $\Delta_d$ is of the form $H(\cdot,\bv)$, 
polynomials are dense in $C(\Delta_d)$, 
and each $T_t$ is a contraction, 
both properties extend to all $f\in C(\Delta_d)$.
Hence $(T_t)_{t\geq0}$ is a Feller semigroup.

Finally, using the martingale property,
$T_tf-f=\int_0^tT_s\cA f\,\dd s$ for $f\in C^2(\Delta_d)$.
Since $\cA f\in C(\Delta_d)$ and $(T_t)_{t\geq0}$ is strongly continuous, this shows that its generator acts on such $f$ as~$\cA f$.
\end{proof}

\section{Proof of the long-term behaviour}\label{sec:longterm}

\subsection{$\bb(1)=0$: Proof of Theorem \ref{fixext}} \label{sec:prooffixation}
In this section we assume the underlying parameter set~$\cP$ to be such that $\bb(1)=0$. 

We start with the following lemma, which provides some basic properties of the operators introduced in Definition~\ref{def:bramutoperators}.
Recall that $\obv(i;k)\in \bM(k)$ is the function that assigns~$1$ to $k\be{i}$ and~$0$ else.

\begin{lemma}\label{lem:boundary_i_nomutation}
Fix~$\mathcal{P}$ with $\bb(1)=0$.
For all $\bv\in \bM$ and $\bw\in \cS_{\cP}(\bv)$: $\bw(\dim(\bw)\be{i})=\bv(\dim(\bv)\be{i})$, $i\in \llbracket d+1\rrbracket$. 
Fix $i\in \llbracket d+1\rrbracket$ and $n,k\in \N$.
For any $\bv\in \cS_{\cP}(\obv(i;k))\cap\bM(n)$:
\begin{enumerate}
    \item $(C^{n,\ell}(\bv))(\balpha)=0$ for all $\balpha\in \Pi(n-\ell+1)$ with $\balpha(i)=0$, $\ell\in \llbracket 2,n\rrbracket$.
    \item $(D^{n,m}(\bv))(\balpha)=0$ for all $\balpha\in \Pi(n+m-1)$ with $\balpha(i)=0$, $m\in \llbracket 2,\kappa\rrbracket$.
    \item $H(\bx,\bv)=0$ for all $\bx\in \partial_{i}\Delta_{d}$.
    \item $H(\be{i},\bv)=\bv(\dim(\bv)\be{i})$.
\end{enumerate}
\end{lemma}
\begin{proof}
    The first claim follows in a straightforward way from the definition of the operators and~\eqref{eq:cond_probability_0}.
    Fix $i\in \llbracket d+1\rrbracket$ and $n,k\in \N$.
    We prove (1) and (2) by induction on the number of concatenated operators to produce some~$\bw\in \bM$ from~$\obv(i;k)$.
    For the base case, it suffices to observe that $\obv(i;k)$ only assigns a positive value to $k\cdot\be{i}$ and thus (1)--(2) are true in this case.
    For the induction step, assume the property holds for some~$\bv\in\cS_{\cP}(\obv(i;k))$, we now show that it also holds for (i) $\bw = C^{n,\ell}(\bv)\in \bM(n-\ell+1)$ and (ii) $\bw = D^{n,m}(\bv)\in \bM(n+m-1)$.
    For (i), let $\balpha\in \Pi(n-\ell+1)$ with $\balpha(i)=0$.
    By the definition of~$C^{n,\ell}$ in \eqref{eq:bernstein_diff_operator}, in the evaluation of $\bw(\balpha)$ the expectation $\mathfrak{E}[\bv(\balpha+(\ell-1)\, \HG^{\balpha}_1)] $ has to be computed.
    To this end, we require the values of $\bv$ assigned to $(\balpha+(\ell-1)\be{j})$, $j\in \llbracket d+1\rrbracket$.
    Since for $j\neq i$, it holds that $(\balpha+(\ell-1)\be{j})(i)=\balpha(i)=0$; 
    in particular by the induction assumption, $\bv(\balpha+(\ell-1)\be{j})=0$ 
    and thus $\mathfrak{E}[\bv(\balpha+(\ell-1)\, \HG^{\balpha}_1) \mid \HG^{\balpha}_1=\be{j}]=0$.
    The probability for the event $\{\HG^{\balpha}_1=\be{i}\}$ is $\balpha(i)/(n-\ell+1)$ and so its contribution to the expectation is also~$0$.
    Thus, $\bw(\balpha)=0$ and part (i) is proved.
    For part (ii), let $\balpha\in \Pi(n+m-1)$ with $\balpha(i)=0$.
    By the definition of~$D^{n,m}$ in~\eqref{eq:bernstein_drift_operator}, evaluating $\bw(\balpha)$ involves computing $\mathfrak{E}[\sum_{j=1}^{d+1}\bp(\HG^{\balpha}_m;j)\, \bv(\balpha-\HG^{\balpha}_m+\be{j})]$.
    Consider the event $\{\HG^{\balpha}_m= \bbeta\}$ for some~$\bbeta\leq\balpha$ (with $\lvert \bbeta \rvert =m$).
    It then holds that~$\bbeta(i)=0$ (because~$\balpha(i)=0$ by assumption).
    To compute the conditional expectation on such an event,
    $\bv$ has to be evaluated on $\balpha-\bbeta+\be{j}$, $j\in \llbracket d+1\rrbracket$. 
    For $j\neq i$, $(\balpha-\bbeta+\be{j})(i)=(\balpha-\bbeta)(i) = \balpha(i)=0$ and thus~$\bv(\balpha-\bbeta+\be{j})=0$ by the induction assumption.
    The summand corresponding to $j=i$, has prefactor $\bp(\bbeta;i)$, which is~$0$ because of~\eqref{eq:cond_probability_0}.
    Since $\HG^{\balpha}_m$ is supported on $\cup_{\bbeta\leq \balpha, \lvert \bbeta \rvert= m}\{\HG^{\balpha}_m= \bbeta\}$,
    also $\bw(\balpha)=0$. 
    This proves~(ii).
    The proof for (1) and (2) is thus completed.
    For~(3), let $\bv\in\cS_{\cP}(\obv(i;k))\cap\bM(n)$ and $\bx\in \partial_{i}\Delta_{d}$.
    For all $\balpha$ with $\balpha(i)>0$, it holds that $\bv(\balpha) \bx(i)^{\balpha(i)}\prod_{j\neq i}\bx(j)^{\balpha(j)}=0$, because $\bx(i)=0$.
    For all $\balpha$ with $\balpha(i)=0$ it was just proved that $\bv(\balpha)=0$ and thus also $\bv(\balpha) \bx^{\balpha}=0$.
    This proves (3).
    Finally, (4) follows in a straightforward way from the definition of~$H$.
    This completes the proof.
    
\end{proof}

The first part of the previous lemma implies that $(\bV_r(L_r\be{i}):i\in \llbracket d+1 \rrbracket)$ is constant under $\bb(1)=0$.
Thus, the positive recurrence of~$L$ under~\eqref{assumption:recurrence_L} translates to~$\bV$.
We now make this precise.

\begin{proof}[Proof of Proposition~\ref{prop:asympV}] 
The proof adapts~\cite[Proof of Prop.~2.3]{Cordero_2026} to the multi-dimensional setting.
Assume $\bb(1)=0$.
Lemma~\ref{lem:boundary_i_nomutation} yields that for all $i\in \llbracket d+1\rrbracket$ the function $r\mapsto \bV_{r}(L_r\be{i})$ is constant. 
In particular, for $\hat{\bv}\in \bM(1)$ the restriction of $\bV$ to the set $\cS_{\cP}(\hat{\bv})$ as defined in~\eqref{defcvv} is an irreducible continuous-time Markov chain.
Indeed, for $\bw,\bw'\in \cS_{\cP}(\hat{\bv})$ to go from $\bw$ to $\bw'$, 
first go from $\bw$ to $\hat{\bv}$ by successive applications of $C^{\cdot,\cdot}$. 
Then go from $\hat{\bv}$ to $\bw'$ in a finite number of successive applications of $C^{\cdot,\cdot}$ and $D^{\cdot,\cdot}$.
 Let us now assume that~\eqref{assumption:recurrence_L} is satisfied. 
	\smallskip
	
	(1) By Lemma~\ref{lem:L_rec_abs}, \eqref{assumption:recurrence_L} implies that $1$ is positive recurrent for $L$.
    Since for every $i\in \llbracket d+1 \rrbracket$ the function $r\mapsto \bV_{r}(L_r\be{i})$ is constant, 
    the positive recurrence of~$1$ for~$L$ translates to the positive recurrence of~$\hat{\bv}$ for~$\bV$.
    By irreducibility of~$\bV$ on $\cS_{\cP}(\hat{\bv})$, the process $\bV$ is positive recurrent, 
    and hence it admits a unique invariant distribution $\mu^{\hat{\bv}}$~\cite[Thm. 3.5.2, Thm. 3.5.3]{norris1998markov} with support $\cS_{\cP}(\hat{\bv})$. 

	\smallskip
	
	(2) Let ${\bV}_{\infty}^{\hat{\bv}}\sim \mu^{\hat{\bv}}$.
    If $\bV_0=\hat{\bv}$, then ${\bV}_r\to {\bV}_{\infty}^{\hat{\bv}}$ in law as $r\to\infty$ by classic Markov chain theory~\citep[Thm. 3.6.2]{norris1998markov}.
    On the other hand, for $\bv \in \bM$ with $\bv(\dim(\bv)\be{i})=\hat{\bv}(\be{i})$, 
    because~$1$ is positive recurrent for~$L$, 
    the process $\bV$ starting at $\bV_0=\bv$ enters $\cS_{\cP}(\hat{\bv})$ in finite time $\P_{\bv}$-almost surely 
    (in particular $\cS_{\cP}(\hat{\bv}) \subset \cS_{\cP}(\bv)$).
    Hence, the result of Proposition~\ref{prop:asympV} follows. 
\end{proof}

Recall that $\obv(i)=\obv(i;1)\in \bM(1)$, i.e. it is the function that assigns~$1$ to $\be{i}$ and~$0$ else.

\begin{lemma} \label{lemf}
Assume $\bb(1)=0$ and that \eqref{assumption:recurrence_L} holds.
For $i\in \llbracket d+1 \rrbracket$,
let $h^{(i)}:\Delta_d\rightarrow [0,\infty)$ be defined by $h^{(i)}(\bx)\coloneqq \E[H(\bx,\bV^{\obv(i)}_{\infty})]$, where the random variable $\bV^{\obv(i)}_{\infty}$ is as in Proposition~\ref{prop:asympV}.
For any $\bx \in \Delta_d$ and $k \in  \N$, we have 
\begin{align}
h^{(i)}(\bx)=\lim_{r\to\infty} \E_{\obv(i;k)}[ H(\bx,\bV_r)], \label{feqlimhxvt}
\end{align}
where the limit in the right-hand side does not depend on the choice of $k \geq 1$. 
Moreover, $h^{(i)}$ is continuous and satisfies for $\bx\in\partial_{i}\Delta_{d}$ that $h^{(i)}(\bx)=0$, $h^{(i)}(\be{i})=1$ and for $\bx\in \Delta_d\setminus(\partial_i\Delta_d\cup\{\be{i}\})$ we have $h^{(i)}(\bx) \in (0,1)$.
\end{lemma}

\begin{proof}
The proof adapts~\cite[Proof of Lem.~5.2]{Cordero_2026} to the multi-dimensional setting.
By Proposition~\ref{prop:asympV}, if $\bV_0=\obv(i;k)$, $\bV_r$ converges in distribution to $\bV^{\obv(i)}_{\infty}$ for any $k\in \N$.
By \eqref{int1} we have that, for any $(\bx,\bv) \in \Delta_d \times \cS_{\cP}(\obv(i;k))$, $\lvert H(\bx,\bv)\rvert \leq 1$.
By the continuous mapping theorem we get \eqref{feqlimhxvt}.
Since $\bV^{\obv(i)}_{\infty}$ is supported on $\cS_{\cP}(\obv(i))$ by Proposition~\ref{prop:asympV}, the mapping $\bx \mapsto H(\bx,\bv)$ is continuous for any $\bv \in \cS_{\cP}(\obv(i))$, so that by dominated convergence we get that also $h^{(i)}$ is continuous.

For $\by\in \Delta_d$ with $\by(i)=0$, we have $h^{(i)}(\by)=0$ by Lemma~\ref{lem:boundary_i_nomutation}-(3); and by Lemma~\ref{lem:boundary_i_nomutation}-(4): $h^{(i)}(\be{i})=1$.
We now justify that $h^{(i)}(\bx) \in (0,1)$ for $\bx\in \Delta_d\setminus(\partial_i\Delta_d\cup\{\be{i}\})$.
First note that for such $\bx$, we have $\bx(i)\in (0,1)$; and thus $H(\bx,\obv(i))=\bx(i)\in (0,1)$.
Moreover, by Proposition \ref{prop:asympV}, we have that $\obv(i)$ is in the support of $\mu^{\obv(i)}$, i.e. $\P(\bV^{\obv(i)}_{\infty}=\obv(i))>0$.
Thus, 
\begin{align}
h^{(i)}(\bx) = \bx(i) \, \P(\bV^{\obv(i)}_{\infty}=\obv(i)) + \sum_{\bv \in \cS_{\cP}(\obv(i)) \setminus \{\obv(i)\}} H(\bx,\bv) \, \P(\bV^{\obv(i)}_{\infty}=\bv). \label{combcvx}
\end{align}
For any $(\bx,\bv )\in \Delta_d \times \cS_{\cP}(\obv(i))$ we have $H(\bx,\bv)\in [0,1]$, because of Lemma~\ref{lem:actionV} and because whenever $\bM\ni\bw\geq 0$, then $C^{n,k}(\bw)\geq 0$ and $D^{n,m}(\bw)\geq 0$.
Since $\sum_{\bv \in \cS_{\cP}(\obv(i))} \P(\bV^{\obv(i)}_{\infty}=\bv)=1$ and $\P(\bV^{\obv(i)}_{\infty}=\obv(i))>0$, we get from \eqref{combcvx} that $h^{(i)}(\bx) \in (0,1)$ for $\bx$ with $\bx(i) \in (0,1)$. 
\end{proof}

\begin{proof}[Proof of Theorem \ref{fixext}]
The proof adapts~\cite[Proof of Thm.~2.4]{Cordero_2026} to the multi-dimensional setting.
Let $h^{(i)}$ be as in Lemma~\ref{lemf}.
Assume $\bX_0=\bx$ for some $\bx\in\mathrm{int}(\Delta_d)$.
First, we prove that $(h^{(i)}(\bX_t))_{t\geq 0}$ is a bounded martingale.
To this end, we use the Markov property of $\bX$, Lemma~\ref{lemf}, dominated convergence (which can be used since, by~\eqref{int1}, we have $\lvert \E_{\obv(i)}[H(\tilde \bX_s,\bV_r)]\rvert\leq \lVert \obv(i)\rVert_\infty=1$), Fubini's theorem, Bernstein duality (Thm.~\ref{thm:Bernsteinduality}), and the Markov property of $\bV$, to obtain 
	\begin{align*}
		\E_{\bx}&[h^{(i)}(\bX_{t+s})\mid \mathcal{F}_t]= \E_{\bX_t}[h^{(i)}(\tilde \bX_{s})]=\E_{\bX_t}[ \lim_{r\to\infty} \E_{\obv(i)}[H(\tilde \bX_s,\bV_r)]]\\
		& =\lim_{r\to\infty} \E_{\bX_t}[ \E_{\obv(i)}[H(\tilde \bX_s,\bV_r)]] = \lim_{r\to\infty}\E_{\obv(i)}[\E_{\bX_t}[H(\tilde \bX_s,\bV_r)]]\\
		&= \lim_{r\to\infty}\E_{\obv(i)}[\E_{\bV_r}[H(\bX_t,\tilde \bV_s)]]=\lim_{r\to\infty} \E_{\obv(i)}[H(\bX_t,\tilde \bV_{r+s})]=h^{(i)}(\bX_t).
	\end{align*}
By Doob's martingale convergence theorem, $F\coloneqq\lim_{t\to\infty} h^{(i)}(\bX_t)$ exists almost surely.
Using Bernstein duality (Thm.~\ref{thm:Bernsteinduality}) and Lemma~\ref{lemf} we get $\E_{\bx}[\bX_t^k(i)]=\E_{\bx}[H(\bX_t,\obv(i;k))]=\E_{\obv(i;k)}[ H(\bx,\bV_t)] \to h^{(i)}(\bx)$ as $t\to\infty$. 

Therefore the positive integer moments of $\bX_t(i)$ converge to the moments of the distribution $\berno{h^{(i)}(\bx)}$.
Probability distributions on $[0,1]$ are completely determined by their positive integer moments, and convergence of those moments imply convergence in distribution.
We thus get that the law of $\bX_t(i)$ converges weakly to $\berno{h^{(i)}(\bx)}$.
Since $\bX_t(i)\to\berno{h^{(i)}(\bx)}$ in law, 
$\min\{\bX_t(i),1-\bX_t(i)\}\to0$ in probability.
By Lemma~\ref{lemf}, 
$h^{(i)}$ is uniformly continuous on $\Delta_d$ with $h^{(i)}=0$ on $\partial_i\Delta_d$ and $h^{(i)}(\be{i})=1$.
Hence $h^{(i)}(\bX_t)-\ind{\bX_t(i)>1/2}\to0$ in probability,
and the law of $h^{(i)}(\bX_t)$ converges weakly to $\berno{h^{(i)}(\bx)}$.
Combining with the above almost sure convergence of $h^{(i)}(\bX_t)$, we get $F \sim \berno{h^{(i)}(\bx)}$; 
so in particular $F \in \{0,1\}$ a.s.
Let us show that $\bX_t(i)$ converges to $F$ almost surely.
Let $\varepsilon \in (0,1/2)\cap \mathbb{Q}$.
Since, by Lemma~\ref{lemf}, $h^{(i)}$ is continuous, positive on $\Delta_d\setminus(\partial_i\Delta_d)$, and strictly smaller than $1$ on $\Delta_d\setminus\{\be{i}\}$,
we have 
\begin{equation*}
	\delta(\varepsilon)\coloneqq\min \{ \min_{\substack{\bz \in \Delta_d \\ \bz(i)\in [\varepsilon,1]}}h^{(i)}(\bz),1-\hspace{-0.4cm}\max_{
	\substack{\bz \in \Delta_d \\ \bz(i)\in [0,1-\varepsilon]}}h^{(i)}(\bz) \}>0.
\end{equation*}
Since $h^{(i)}(\bX_t)$ converges to $F$ almost surely, there exists almost surely $T(\varepsilon)$ such that $\lvert h^{(i)}(\bX_t)-F\rvert<\delta(\varepsilon)$ for all $t \geq T(\varepsilon)$.
Distinguishing the cases $F=0$ and $F=1$ we see that this is possible only if $\lvert \bX_t(i)-F\rvert <\varepsilon$ for all $t \geq T(\varepsilon)$.
In conclusion, $\P(\lvert \bX_t(i)-F\rvert<\varepsilon \ \text{for all large }t)=1$.
Since this is true for all $\varepsilon \in (0,1/2)\cap \mathbb{Q}$ we get that $\bX_t(i)$ converges almost surely to $F$.
This proves that $\bX_\infty(i)=F$.
Repeating the argument for every $i\in \llbracket d+1\rrbracket$, and noting that using the duality we have
\begin{equation*}
    \sum_{i=1}^{d+1}h^{(i)}(\bx)=\lim_{r\to\infty}\sum_{i=1}^{d+1}\E_{\bx}[H(\bX_r,\obv(i))]=\lim_{r\to\infty}\sum_{i=1}^{d+1}\E_{\bx}[\bX_r(i)]=1,
\end{equation*}
yields $(\bX_r(1),\ldots \bX_r(d))\xrightarrow{r\to\infty}\bX_\infty$ almost surely with $\bX_{\infty}$ taking values in $\{\be{i}\in \Delta_d: i\in \llbracket d+1\rrbracket \}$, and each such value is taken with a positive probability.
Finally, \eqref{exprfivproba} follows from $F \sim \berno{h^{(i)}(\bx)}$ and Lemma~\ref{lemf}. 
\end{proof}

\subsection{$\bb(1)>0$: Proof of Theorem~\ref{thstadist}}\label{sec:proofstationary}
In this section, we study the ergodicity of~$\bX$ and characterise its stationary distribution, both under the assumption that~$\bb(1)>0$, $P$ is irreducible, and 
either~$\lvert \btheta\rvert>0$, which is~\eqref{cond:stationary_parent_independent}, 
or $\lvert \btheta\rvert=0$, $\kappa=1$, and $P$ being in addition aperiodic, which is~\eqref{cond:stationary_true_parent_dependent}. 

\begin{proof}[Proof of Proposition~\ref{prop:absV}]
    We first prove the result under~\eqref{cond:stationary_parent_independent}.
    To this end, we can adapt~\cite[Proof of Thm.~2.6]{Cordero2021} to the multi-dimensional setting.
    In this case, by Lemma~\ref{lem:L_rec_abs}-(2), $L$ absorbs in finite time in~$0$.
    Let $\tau_0\coloneqq \inf\{t\geq 0:\, L_t=0\}.$
    For $r\geq \tau_0$, there are no transitions of~$\bV$ and thus $\bV_r=\bV_{\tau_0}\in \bM(0)$ for all $r\geq \tau_0$.
    In this case set $U_{\infty}\coloneqq\bV_{\tau_0}(0,\ldots,0)$.
    In particular, $\min_{\Pi(L_r)}\bV_r(\balpha)=\max_{\Pi(L_r)}\bV_r(\balpha)=U_{\infty}$ for all $r\geq \tau_0$, $\P_{\bv}$-almost surely.
    This completes the proof of the proposition under~\eqref{cond:stationary_parent_independent}.
    
    Next, assume~\eqref{cond:stationary_true_parent_dependent}.
    In this case, by Lemma~\ref{lem:L_rec_abs}-(1), $1$ is positive recurrent for~$L$.
    In fact, it follows from the definition of~$\bV$ that because~$\kappa=1$, the state $1$ is absorbing for~$L$.
    Let $\tau_1\coloneqq \inf\{t\geq 0:\, L_t=1\}$.
    For $r\geq \tau_1$, $\bV_r\in \bM(1)$.
    Note that $\bM(1)$ is isomorphic to $\R^{d+1}$ via the map $\bv\mapsto (\bv(\be{1}),\ldots, \bv(\be{d+1}))\eqqcolon \vec{\bv}$ (interpreted as a row vector).    
    Applying $D^{1,1}$ to $\bv$ corresponds to multiplying $\vec{\bv}$ from the right with the matrix~$Q\coloneqq P^{\top}$ (recall $\tilde P=P$ since $\lvert\btheta\rvert=0$), whose columns sum to~$1$.
    In particular, $\lim_{r\to\infty}\bV_{\tau_1+r}=\lim_{n\to\infty} \vec{\bV}_{\tau_1}Q^{n}$, 
    with $\vec{\bV}_{\tau_1}$ being the (random) vector that arises the first time~$L$ hits~$1$.
    $Q$ is primitive (since $P$ is) and has, 
    by Perron--Frobenius theory, left eigenvector $(1,\ldots,1)$ and positive right eigenvector~$\psi$, normalised such that $\lvert \psi \rvert =1$. 
    That is, $\psi^\top$ is the stationary distribution of~$P$.
    By~\cite[Thm.~1.2]{Seneta2006}, $Q^n\xrightarrow{n\to\infty}\psi\cdot(1,\ldots,1)$.
    For later reference, we note that this convergence is exponential.
    In particular, $\lim_{t\to\infty}\bV_t = \vec\bV_{\tau_1}\cdot\psi\cdot(1,\ldots,1)$ and $U_{\infty}=\vec\bV_{\tau_1}\cdot\psi$.
    This completes the proof under~\eqref{cond:stationary_true_parent_dependent}.
\end{proof}

\begin{proof}[Proof of Theorem \ref{thstadist}]
We start out by proving that all mixed integer moments of $\bX_t$ converge under the stated assumptions to a limit independently of the initial condition~$\bx$ of $\bX$.
Fix some exponent $\balpha\in \Pi(n)$ for some $n\in \N_0$. 
First, consider the case~\eqref{cond:stationary_parent_independent}.
Let again $\tau_0 \coloneqq \inf \{ r \geq 0:\, L_r=0 \}$ and recall also from the previous proof that $\bV_{\tau_0}(0,\ldots,0)=U_{\infty}$.
Moreover, 
by the Bernstein duality (Thm.~\ref{thm:Bernsteinduality}) and~\eqref{int1},
\begin{align}
\lvert  \E_{\bx}[\bX_t^{\balpha}] - \E_{\balpha}[ U_\infty] \rvert \leq &\lvert \E_{\balpha}[ H(\bx,\bV_t) - H(\bx,\bV_{\tau_0}) ] \rvert \nonumber \\
\leq  &2 \, \P_{\balpha} (\tau_0 > t) \xrightarrow[]{t \rightarrow \infty} 0.\label{eq:bound_tau0} 
\end{align}
Next, assume~\eqref{cond:stationary_true_parent_dependent}.
Recall the notation from the previous proof.
In particular, let again $\tau_1 \coloneqq \inf \{ r \geq 0:\, L_r=1 \}$
and recall that for every~$j\in \llbracket d+1\rrbracket$, 
\[\lvert\bV_{\tau_1+r}(\be{j}) - (\bV_{\tau_1}(\be{k}):\, k\in \llbracket d+1 \rrbracket ) \cdot \psi\rvert \xrightarrow{r\to\infty}0,\] 
exponentially,
where $\psi$ is the normalised right eigenvector of~$Q$ from the previous proof.
Using again the Bernstein duality (Thm.~\ref{thm:Bernsteinduality}) and~\eqref{int1}, but also the just-mentioned exponential convergence and Lemma~\ref{lem:L_rec_abs},
\begin{align}
&\lvert  \E_{\bx}[\bX_t^{\balpha}] - \E_{\balpha}[ U_\infty] \rvert \nonumber \\
\leq\ & 2\P_{\balpha}(t< \tau_1)+\E_{\balpha}\Big[ \max_{j\in \llbracket d+1\rrbracket} \lvert\bV_{t}(\be{j}) - U_\infty\rvert\,\ind{t\geq \tau_1}\Big] \xrightarrow[]{t \rightarrow \infty} 0. \label{eq:bound_tau1}
\end{align}
Thus we have shown that in both cases, any mixed integer moment of $\bX_t$ converges to a limit (which does not depend on the initial condition $\bx$) as $t\to\infty$, that is,  
\begin{equation}
    \forall{\bx\in \Delta_d}:\quad \lim_{t\to\infty}\E_{\bx}[\bX_t^{\balpha}] = \E_{\balpha}[U_{\infty}].
    \label{eq:limit_moments}
\end{equation}
Using the Stone--Weierstrass theorem, one can show that any probability distribution on a compact set is uniquely determined by its mixed moments.
In particular, there is a probability distribution $\mathcal{L}$ on $\Delta_d$ such that for any $\bx\in\Delta_d$, the law of $\bX_t$ under $\P_{\bx}$ converges in distribution to $\mathcal{L}$ as $t\to\infty$, and the mixed moments of this law are given by~\eqref{momentsloiinvar}. 
By dominated convergence, the convergence of the law of $\bX_t$ towards $\mathcal{L}$ as $t\to\infty$ extends to any distribution on $\Delta_d$ for $\bX_0$.
Hence, $\mathcal{L}$ is the unique stationary distribution of~$\bX$.

It remains to prove that for all $i$, $\P_{\bx}(\bX_t\xrightarrow{t\to\infty}\be{i})=0$.
Since $\bX_t\in \Delta_d$, $\P_{\bx}(\lim_{t\to\infty}\bX_t=\be{i})=0\Leftrightarrow\P_{\bx}(\lim_{t\to\infty} \bX_t(i)=1)=0$.
    We first show that $\E_{\obv(j)}[U_\infty]>0$ for every $j\in\llbracket d+1\rrbracket$.
    Under~\eqref{cond:stationary_true_parent_dependent}, starting from $\obv(j)$ we have $L_0=1$, hence $\tau_1=0$ and $U_\infty=\psi(j)>0$ by the proof of Proposition~\ref{prop:absV}.
    Under~\eqref{cond:stationary_parent_independent}, let $I_{\theta}\coloneqq\{k\in\llbracket d+1\rrbracket:\btheta(k)>0\}$.
    By irreducibility of~$P$, there are $n\in\N_0$ and $k\in I_\theta$ with $\tilde P^n_{k,j}>0$ (take a shortest path in~$P$ from $I_\theta$ to~$j$; its edges after the first vertex lead to states outside~$I_\theta$, on which $\tilde P$ and $P$ agree).
    Under $\P_{\obv(j)}$, with positive probability, the first $n$ transitions of $\bV$ are of type $D^{1,1}$ and the $(n+1)$st is of type~$D^{1,0}$; on this event
    $U_\infty=\sum_{k}\frac{\btheta(k)}{\lvert\btheta\rvert}\big((\tilde P/(1-\lvert\btheta\rvert))^n\big)_{k,j}>0$.
    Since $U_\infty\geq0$ whenever $\bV_0\geq0$, we get $\E_{\obv(j)}[U_\infty]>0$.
    Moreover, by~\eqref{eq:limit_moments} with $\balpha=\be{j}$ (note that $\bv^{\be{j}}=\obv(j)$), we have $\sum_{j=1}^{d+1}\E_{\obv(j)}[U_\infty]=\lim_{t\to\infty}\E_{\bx}[\sum_{j=1}^{d+1}\bX_t(j)]=1$.
    Since each summand is positive, 
    it follows that $\E_{\obv(j)}[U_\infty]\in(0,1)$.
    Finally, we fix~$i$ and consider~$\bV$ started from $\obv(i)$.
    Choose $\varepsilon>0$ such that $(1+\varepsilon)\E_{\obv(i)}[U_\infty]<1$.
    The right-hand sides of~\eqref{eq:bound_tau0} and~\eqref{eq:bound_tau1} do not depend on~$\bx$, 
    so the convergence in~\eqref{eq:limit_moments} is uniform in~$\bx$.
    Hence there is $t_0>0$ with $\E_{\bx}[\bX_{t_0}(i)]<(1+\varepsilon)^{1/2}\E_{\obv(i)}[U_\infty]$ for all $\bx\in\Delta_d$, and by Markov's inequality
    $p\coloneqq\sup_{\bx\in\Delta_d}\P_{\bx}\big(\bX_{t_0}(i)>(1+\varepsilon)\E_{\obv(i)}[U_\infty]\big)\leq(1+\varepsilon)^{-1/2}<1$.
    Set $c\coloneqq(1+\varepsilon)\E_{\obv(i)}[U_\infty]<1$.
    By the Markov property at the times $nt_0$ and induction on~$\ell$,
    $\P_{\bx}\big(\bX_{nt_0}(i)>c \text{ for } n=m+1,\ldots,m+\ell\big)\leq p^{\ell}$ for all $m,\ell\in\N_0$.
    Letting $\ell\to\infty$ shows that, for every~$m$, the event $\{\bX_{nt_0}(i)>c \text{ for all } n\geq m\}$ has probability~$0$.
    Hence, almost surely, $\bX_{nt_0}(i)\leq c<1$ for infinitely many~$n$, and so $\P_{\bx}(\bX_t(i)\xrightarrow{t\to\infty}1)=0$.
    This completes the proof.
\end{proof}
\appendix
\section{Establishing the generator identity: proof of Lemma~\ref{lem:generator_identity}.}\label{sec:generator_identity}

\begin{proof}[Proof of Lemma~\ref{lem:generator_identity}]
    We establish~\eqref{eq:generatorrelation} for each component of the generator separately.
    Recall the definition of the $\Lambda$-Wright--Fisher process generator from~\eqref{eq:wfgenerator}.
    The generator of the Bernstein coefficient process, $\cB$, acts on functions $g:\cS_{\cP}(\bv)\to\R$ via
    ${\cB}g(\bw)=({\cB}_{\bd}+{\cB}_{\Lambda_{>}}+{\cB}_{\Lambda(0)})g(\bw),$
    where for $\bw\in \cS_{\cP}(\bv)\cap\bM(n)$,
	\begin{equation}\label{eq:dual_generator}
    \begin{split}
        {\cB}_{\bd}g(\bw)&\coloneqq n\lvert \btheta\rvert\bb(1) [g(D^{n,0}(\bw))-g(\bw)] \\
        & \qquad + n\sum_{m=1}^{\kappa}\bb(m)(1- \ind{m=1}\lvert \btheta\rvert )[g(D^{n,m}(\bw))-g(\bw)],\\
        {\cB}_{\Lambda_{>}}g(\bw)&\coloneqq \sum_{k=2}^n \binom{n}{k} \lambda_{n,k}^{>} [g({C}^{n,k}(\bw))-g(\bw)]\\
        {\cB}_{\Lambda(0)}g(\bw)&\coloneqq  \Lambda(\{0\})\binom{n}{2} [g({C}^{n,2}(\bw))-g(\bw)],
    \end{split}
	\end{equation}
    where 
    $\lambda^{>}_{n,k}\coloneqq\int_{(0,1]}z^{k-2}(1-z)^{n-k}\Lambda(\dd z)$, so that 
    $\lambda_{n,k}=\lambda^{>}_{n,k}+\Lambda(\{0\})\ind{k=2}$.
    
    We start by proving $\cA_{\bd}H(\cdot,\bw)(\bx)={\cB}_{\bd}H(\bx,\cdot)(\bw)$: 
    {\footnotesize
	\begin{align*}
    &\cA_{\bd}H(\cdot,\bw)(\bx) \\
    &= \sum_{m=1}^{\kappa} \bb(m) \sum_{i=1}^{d}\left\{ \left[-\bx(i) + \sum_{\substack{\bbeta\in \N_0^{d+1}:\\ \lvert \bbeta \rvert = m}} \bp(\bbeta;i) \binom{m}{\bbeta} \bx^{\bbeta}
    \right] \right.\\
    &\left. \left[\sum_{\substack{\balpha\in \N_0^{d+1}: \lvert \balpha \rvert = n \\ \balpha(i)\geq 1}}\bw(\balpha) \binom{n}{\balpha}\balpha(i)\bx^{\balpha-\be{i}}-\sum_{\substack{\balpha\in \N_0^{d+1}: \lvert \balpha \rvert = n \\ \balpha(d+1)\geq 1}}\bw(\balpha) \binom{n}{\balpha}\balpha(d+1)\bx^{\balpha-\be{d+1}}\right]\right\} \\
    &= \sum_{m=1}^{\kappa} \bb(m) \left\{\sum_{\substack{\bbeta\in \N_0^{d+1}:\\ \lvert \bbeta \rvert = m}}  \left[\sum_{i=1}^{d}\left(\sum_{\substack{\balpha\in \N_0^{d+1}: \lvert \balpha \rvert = n \\ \balpha(i)\geq 1}}  \bp(\bbeta;i) \bw(\balpha) \binom{m}{\bbeta}\binom{n}{\balpha}\balpha(i) \bx^{\balpha+\bbeta-\be{i}}\right)\right.\right. \\
    &\left.\left.+ \sum_{\substack{\balpha\in \N_0^{d+1}: \lvert \balpha \rvert = n \\ \balpha(d+1)\geq 1}}  (\bp(\bbeta;d+1)-1) \bw(\balpha) \binom{m}{\bbeta}\binom{n}{\balpha}\balpha(d+1) \bx^{\balpha+\bbeta-\be{d+1}}
    \right]  - \sum_{\balpha\in\Pi(n)}\bw(\balpha)\binom{n}{\balpha}\big(n-\balpha(d+1)\big)\bx^{\balpha} \right. \\
    & \left. +\sum_{\substack{\balpha\in \N_0^{d+1}: \lvert \balpha \rvert = n \\ \balpha(d+1)\geq 1}} \bw(\balpha) \binom{n}{\balpha} \balpha(d+1) \bx^{\balpha-\be{d+1}} (1-\bx(d+1))\right\}\\
    &= \sum_{m=1}^{\kappa}\bb(m)  \left\{\sum_{\substack{\bbeta\in \N_0^{d+1}:\\ \lvert \bbeta \rvert = m}}   \sum_{i=1}^{d+1}\left(\sum_{\substack{\balpha\in \N_0^{d+1}: \\\lvert \balpha \rvert = n-1}}  \bp(\bbeta;i) \bw(\balpha+\be{i}) \binom{m}{\bbeta}\binom{n}{\balpha+\be{i}}(\balpha(i)+1) \bx^{\balpha+\bbeta}\right) -n H(\bx,\bw) \right\} 
\end{align*}
}
where in the second step we used that $ \bp (\bbeta; d+1) = 1-\sum_{i=1}^{d} \bp(\bbeta;i)$ and $\bx(d+1) = 1-\sum_{i=1}^d \bx(i).$
For $m\in\llbracket 2,\kappa\rrbracket$,
the $m$th summand can be rewritten as
$$n \bb(m) \left\{\sum_{\substack{\balpha\in \N_0^{d+1}: \\\lvert \balpha \rvert = n+m-1}} \binom{n+m-1}{\balpha}\bx^{\balpha} \sum_{\substack{\bbeta\in \N_0^{d+1}: \lvert \bbeta \rvert = m \\ \bbeta\leq \balpha}} \frac{\prod_{k=1}^{d+1}\binom{\balpha(k)}{\bbeta(k)}}{\binom{n+m-1}{m}}  \sum_{i=1}^{d+1} \bp(\bbeta;i) \bw(\balpha-\bbeta+\be{i})  - H(\bx,\bw) \right\},$$
which leads to the form required in the definition of $D^{n,m}$.
For $m=1$ and $\lvert \btheta\rvert\notin\{0,1\}$, 
the summand can be rewritten as 
{\footnotesize
\begin{align*}
    &n\bb(1) \left(\sum_{j=1}^{d+1}   \sum_{i=1}^{d+1}\sum_{\substack{\balpha\in \N_0^{d+1}: \\\lvert \balpha \rvert = n-1}}  \bp(\be{j};i) \bw(\balpha+\be{i}) \binom{n-1}{\balpha}\bx^{\balpha+\be{j}}\right) -n\bb(1) H(\bx,\bw)\\
    &=(1-\lvert \btheta\rvert) \ n\bb(1)\left(\sum_{\substack{\balpha\in \N_0^{d+1}: \\\lvert \balpha \rvert = n}} \binom{n}{\balpha} \bx^{\balpha}\sum_{j=1}^{d+1}   \sum_{i=1}^{d+1}  \frac{\tilde{\bp}(\be{j};i)}{1-\lvert \btheta\rvert} \frac{\balpha(j)}{n}  \bw(\balpha+\be{i}-\be{j})\right) \\
    & + \lvert \btheta\rvert n\bb(1)\left(\sum_{\substack{\balpha\in \N_0^{d+1}: \\\lvert \balpha \rvert = n-1}}  \binom{n-1}{\balpha}\bx^{\balpha}\sum_{i=1}^{d+1} \frac{\btheta(i)}{\lvert \btheta\rvert} \bw(\balpha+\be{i}) \right) -n \bb(1)H(\bx,\bw)\\
    &=n(1-\lvert \btheta\rvert) \ \bb(1) [H(\bx,D^{n,1}(\bw))-H(\bx,\bw)] + n \lvert \btheta\rvert \ \bb(1) [H(\bx,D^{n,0}(\bw))-H(\bx,\bw)].
\end{align*}
}
It is not difficult to check that also for $\lvert \btheta\rvert\in\{0,1\}$ the above equation holds.
Altogether this proves $\cA_{\bd}H(\cdot,\bw)(\bx)={\cB}_{\bd}H(\bx,\cdot)(\bw)$.

\medskip

Next, we establish $\cA_{\Lambda_>}H(\cdot,\bw)(\bx)={\cB}_{\Lambda_>}H(\bx,\cdot)(\bw)$:
{\footnotesize
\begin{align*}
    &\cA_{\Lambda_>}H(\cdot,\bw)(\bx) \\
    &= \int_{(0,1]} \sum_{i=1}^{d+1}\bx(i)\sum_{\substack{\balpha\in \N_0^{d+1}: \\ \lvert \balpha \rvert = n}}\bw(\balpha)\binom{n}{\balpha}\left[ ((1-r)\bx(i)+r)^{\balpha(i)} \prod_{\substack{j\in\llbracket d+1\rrbracket: \\ j\neq i}} (\bx(j)(1-r))^{\balpha(j)} - \bx^{\balpha} \right]\frac{\Lambda(\dd r)}{r^2}\\
    &= \int_{(0,1]} \left[\sum_{\substack{\balpha\in \N_0^{d+1}: \\ \lvert \balpha \rvert = n}}
    \sum_{i=1}^{d+1}
    \sum_{k=0}^{\balpha(i)}\binom{\balpha(i)}{k}r^k(1-r)^{n-k}\bw(\balpha)\binom{n}{\balpha} \bx^{\balpha-(k-1)\be{i}}
    - H(\bx,\bw) \right]\frac{\Lambda(\dd r)}{r^2}\\
    &= \int_{(0,1]} \left[\sum_{\substack{\balpha\in \N_0^{d+1}: \\ \lvert \balpha \rvert = n}}
    \sum_{i=1}^{d+1}
    \sum_{k=2}^{\balpha(i)}\binom{\balpha(i)}{k}r^k(1-r)^{n-k}\bw(\balpha)\binom{n}{\balpha} \bx^{\balpha-(k-1)\be{i}}
    - (1-(1-r)^n-nr(1-r)^{n-1}) H(\bx,\bw) \right]\frac{\Lambda(\dd r)}{r^2}\\
    &= \sum_{k=2}^{n} \binom{n}{k} \lambda_{n,k}^{>} \left[\sum_{i=1}^{d+1} \sum_{\substack{\balpha\in \N_0^{d+1}: \lvert \balpha \rvert = n \\ \balpha(i)\geq k}} \frac{\binom{\balpha(i)}{k} \binom{n}{\balpha}}{\binom{n}{k}}\bw(\balpha)\bx^{\balpha-(k-1)\be{i}}-H(\bx,\bw)\right]\\
    &= \sum_{k=2}^{n} \binom{n}{k} \lambda_{n,k}^{>}
    \left[\sum_{\substack{\balpha\in \N_0^{d+1}: \lvert \balpha \rvert = n-k+1}} \binom{n-k+1}{\balpha} \bx^{\balpha} 
    \sum_{i=1}^{d+1}\bw(\balpha+(k-1)\be{i}) \frac{\binom{n}{\balpha+(k-1)\be{i}} \binom{\balpha(i)+k-1}{k}}{\binom{n}{k} \binom{n-k+1}{\balpha}} -H(\bx,\bw)\right]\\
    &={\cB}_{\Lambda_>}H(\bx,\cdot)(\bw),
\end{align*}
}
where in the fourth step we interchanged sum and integral,
which is justified since all sums are finite and 
$r^{k-2}(1-r)^{n-k}$ is bounded on $(0,1]$ for $k\geq2$.

Finally, we establish $\cA_{\Lambda(0)}H(\cdot,\bw)(\bx)={\cB}_{\Lambda(0)}H(\bx,\cdot)(\bw)$.
Without loss of generality, we assume in the following that $\Lambda(\{0\})=1$, and we compute $2\cA_{\Lambda(0)}H$, i.e.\ we omit the factor $1/2$ in~\eqref{eq:wfterm}.
{\footnotesize
\begin{align*}
    &2\cA_{\Lambda(0)}H(\cdot,\bw)(\bx) \\
    &= \sum_{i=1}^d \left\{ \sum_{\substack{\balpha\in \N_0^{d+1}: \lvert \balpha \rvert = n \\ \balpha(i)\geq 1}} \bw(\balpha)\binom{n}{\balpha} \balpha(i)(\balpha(i)-1) (1-\bx(i))\bx^{\balpha-\be{i}} \right.\\
    &\left.- 2\sum_{\substack{\balpha\in \N_0^{d+1}: \lvert \balpha \rvert = n \\ \balpha(d+1)\geq 1}} \bw(\balpha)\binom{n}{\balpha} \balpha(i)\balpha(d+1) (1-\bx(i))\bx^{\balpha-\be{d+1}}\right.\\
    &\left.+ \sum_{\substack{\balpha\in \N_0^{d+1}: \lvert \balpha \rvert = n \\ \balpha(d+1)\geq 2}} \bw(\balpha)\binom{n}{\balpha} \balpha(d+1)(\balpha(d+1)-1) (1-\bx(i))\bx^{\balpha+\be{i}-2\be{d+1}} \right. \\
    &+ \sum_{\substack{j=1 \\ j\neq i}}^d\left[-\sum_{\substack{\balpha\in \N_0^{d+1}: \\ \lvert \balpha \rvert = n}} \bw(\balpha)\binom{n}{\balpha}\balpha(i)\balpha(j) \bx^{\balpha} 
    +2\sum_{\substack{\balpha\in \N_0^{d+1}: \lvert \balpha \rvert = n \\ \balpha(d+1)\geq 1}}\bw(\balpha)\binom{n}{\balpha} \balpha(j)\balpha(d+1)\bx^{\balpha+\be{i}-\be{d+1}}\right.\\
    &\left. \left. - \sum_{\substack{\balpha\in \N_0^{d+1}: \lvert \balpha \rvert = n \\ \balpha(d+1)\geq 2}}\bw(\balpha)\binom{n}{\balpha}\balpha(d+1)(\balpha(d+1)-1) \bx^{\balpha-2\be{d+1}+\be{i}+\be{j}} \right]\right\} \\
    &= \sum_{i=1}^d \left\{ \sum_{\substack{\balpha\in \N_0^{d+1}: \lvert \balpha \rvert = n \\ \balpha(i)\geq 1}} \bw(\balpha)\binom{n}{\balpha} \balpha(i)(\balpha(i)-1) (1-\bx(i))\bx^{\balpha-\be{i}} \right.\\
    &\left.- 2\sum_{\substack{\balpha\in \N_0^{d+1}: \lvert \balpha \rvert = n \\ \balpha(d+1)\geq 1}} \bw(\balpha)\binom{n}{\balpha} \balpha(i)\balpha(d+1) (1-\bx(i))\bx^{\balpha-\be{d+1}}\right.\\
    &\left.+ \sum_{\substack{\balpha\in \N_0^{d+1}: \lvert \balpha \rvert = n \\ \balpha(d+1)\geq 2}} \bw(\balpha)\binom{n}{\balpha} \balpha(d+1)(\balpha(d+1)-1) (1-\bx(i))\bx^{\balpha+\be{i}-2\be{d+1}} \right. \\
    &-\sum_{\substack{\balpha\in \N_0^{d+1}: \\ \lvert \balpha \rvert = n}} \bw(\balpha)\binom{n}{\balpha}\balpha(i)(n-\balpha(i)-\balpha(d+1)) \bx^{\balpha} \\
    &+2\sum_{\substack{\balpha\in \N_0^{d+1}: \lvert \balpha \rvert = n \\ \balpha(d+1)\geq 1}}\bw(\balpha)\binom{n}{\balpha} (n-\balpha(i)-\balpha(d+1)) \balpha(d+1)\bx^{\balpha+\be{i}-\be{d+1}}\\
    &\left.  - \sum_{\substack{\balpha\in \N_0^{d+1}: \lvert \balpha \rvert = n \\ \balpha(d+1)\geq 2}}\bw(\balpha)\binom{n}{\balpha}\balpha(d+1)(\balpha(d+1)-1) (1-\bx(i)-\bx(d+1))\bx^{\balpha+\be{i}-2\be{d+1}} \right\} \\
    &= \sum_{i=1}^d \left\{ \sum_{\substack{\balpha\in \N_0^{d+1}: \lvert \balpha \rvert = n \\ \balpha(i)\geq 1}} \bw(\balpha)\binom{n}{\balpha} \balpha(i)(\balpha(i)-1) \bx^{\balpha-\be{i}} \right.\\
    &\left.- 2\sum_{\substack{\balpha\in \N_0^{d+1}: \lvert \balpha \rvert = n \\ \balpha(d+1)\geq 1}} \bw(\balpha)\binom{n}{\balpha} \balpha(i)\balpha(d+1) \bx^{\balpha-\be{d+1}}\right.\\
    &\left.+ \sum_{\substack{\balpha\in \N_0^{d+1}: \lvert \balpha \rvert = n \\ \balpha(d+1)\geq 1}} \bw(\balpha)\binom{n}{\balpha} \balpha(d+1)(\balpha(d+1)-1) \bx^{\balpha+\be{i}-\be{d+1}} \right. \\
    &-\sum_{\substack{\balpha\in \N_0^{d+1}: \\ \lvert \balpha \rvert = n}} \bw(\balpha)\binom{n}{\balpha}\balpha(i)(n-1-\balpha(d+1)) \bx^{\balpha} \\
    &\left.+2\sum_{\substack{\balpha\in \N_0^{d+1}: \lvert \balpha \rvert = n \\ \balpha(d+1)\geq 1}}\bw(\balpha)\binom{n}{\balpha} (n-\balpha(d+1)) \balpha(d+1)\bx^{\balpha+\be{i}-\be{d+1}}\right\}\\
    &= \sum_{\substack{\balpha\in \N_0^{d+1}: \lvert \balpha \rvert = n-1}} \sum_{i=1}^d \bw(\balpha+\be{i})\binom{n}{\balpha+\be{i}} (\balpha(i)+1)\balpha(i)\bx^{\balpha} \\
    &- 2\sum_{\substack{\balpha\in \N_0^{d+1}: \lvert \balpha \rvert = n \\ \balpha(d+1)\geq 1}} \bw(\balpha)\binom{n}{\balpha} (n-\balpha(d+1))\balpha(d+1) \bx^{\balpha-\be{d+1}}\\
    &\left.+ \sum_{\substack{\balpha\in \N_0^{d+1}: \lvert \balpha \rvert = n \\ \balpha(d+1)\geq 1}} \bw(\balpha)\binom{n}{\balpha} \balpha(d+1)(\balpha(d+1)-1) (1-\bx(d+1))\bx^{\balpha-\be{d+1}} \right. \\
    &-\sum_{\substack{\balpha\in \N_0^{d+1}: \\ \lvert \balpha \rvert = n}} \bw(\balpha)\binom{n}{\balpha}(n-\balpha(d+1))(n-1-\balpha(d+1)) \bx^{\balpha} \\
    &+2\sum_{\substack{\balpha\in \N_0^{d+1}: \lvert \balpha \rvert = n \\ \balpha(d+1)\geq 1}}\bw(\balpha)\binom{n}{\balpha} (n-\balpha(d+1)) \balpha(d+1)(1-\bx(d+1))\bx^{\balpha-\be{d+1}}\\
    &= n(n-1)\left[ \sum_{\substack{\balpha\in \N_0^{d+1}: \lvert \balpha \rvert = n-1}} \binom{n-1}{\balpha}\bx^{\balpha} \sum_{i=1}^{d+1} \bw(\balpha+\be{i})  \frac{\balpha(i)}{n-1} - \sum_{\substack{\balpha\in \N_0^{d+1}: \lvert \balpha \rvert = n }} \bw(\balpha)\binom{n}{\balpha} \bx^{\balpha} \right] \\
    &=2{\cB}_{\Lambda(0)}H(\bx,\cdot)(\bw).
\end{align*}
}

Altogether, this proves \eqref{eq:generatorrelation}, which completes the proof. 
\end{proof}

\section*{Acknowledgements}

I would like to thank Fernando Cordero for his insights and stimulating discussions throughout the course of this research.

\bibliographystyle{abbrv}
\bibliography{reference.bib}

\end{document}